\documentclass{amsart}
\usepackage{amsmath,amssymb,amsfonts,mathrsfs, amsthm, mathtools}
\usepackage[]{units}
\usepackage{hyperref}
\usepackage{parskip}
\usepackage[utf8]{inputenc}
\usepackage[english]{babel}
\usepackage{lipsum}  
\usepackage{color}
\usepackage{bbm}
\usepackage{stmaryrd}
\usepackage[foot]{amsaddr}
\usepackage{enumitem}
\usepackage{subcaption}
\usepackage{floatrow}
\newfloatcommand{capbtabbox}{table}[][\FBwidth]
\usepackage{blindtext}

\newtheorem{theorem}{Theorem}[section]
\newtheorem{lemma}[theorem]{Lemma}

\theoremstyle{definition}
\newtheorem{definition}[theorem]{Definition}

\theoremstyle{remark}
\newtheorem{remark}[theorem]{Remark}

\DeclarePairedDelimiter\ceil{\lceil}{\rceil}

\begin{document}
	
	\title[$N$-Body Dielectric Spheres Problem. Part II.]{An Integral Equation Formulation of the $N$-Body Dielectric Spheres Problem. Part II: Complexity Analysis}
	
	\author{$\text{B\'erenger Bramas}^*$}
	\address[*]{CAMUS Team, Inria Nancy - Grand Est, ICube - Laboratoire des sciences de l'ing\'enieur, de l'informatique et de l'imagerie.}
	\email{berenger.bramas@inria.fr}
	
	\author{$\text{Muhammad Hassan}^{\dagger}$}
	\address[$\dagger$]{Center for Computational Engineering Science, Department of Mathematics, RWTH Aachen University, Germany}
	\email{hassan@mathcces.rwth-aachen.de}
	
	\author{$\text{Benjamin Stamm}^{\dagger}$}
	\email{stamm@mathcces.rwth-aachen.de}
	
	\date{\today}
	
	
	\maketitle

	\begin{abstract}
		This article is the second in a series of two papers concerning the mathematical study of a boundary integral equation of the second kind that describes the interaction of $N$ dielectric spherical particles undergoing mutual polarisation. The first article presented the numerical analysis of the Galerkin method used to solve this boundary integral equation and derived $N$-independent convergence rates for the induced surface charges and total electrostatic energy. The current article will focus on computational aspects of the algorithm. We provide a convergence analysis of the iterative method used to solve the underlying linear system and show that the number of liner solver iterations required to obtain a solution is independent of $N$. Additionally, we present two linear scaling solution strategies for the computation of the approximate induced surface charges. Finally, we consider a series of numerical experiments designed to validate our theoretical results and explore the dependence of the numerical errors and computational cost of solving the underlying linear system on different system parameters. 
	\end{abstract}
	
	\section{Introduction}
	
	$N$-body problems are ubiquitous in a wide variety of physical fields including quantum mechanics, molecular dynamics, astrophysics, and electrostatics. In the field of chemical physics, $N$-body problems arise naturally when one considers the interaction of a large number of charged particles (see, e.g., \cite{dobnikar2002many,Colloid, Extra2, Crystal, jurrus2018improvements, Extra1, merrill2009many,Titan, Self, Lattice, yap2013calculating, barros2014dielectric}). If the particles are composed of a polarisable dielectric material, then a full description of the electrostatic interaction can typically not be obtained as simply the sum of pairwise Coulomb-type interactions. Consequently, more elaborate numerical methods based either on so-called image charge methods (see, e.g.,\cite{image1, image3, image2, qin2016image}) or multipole expansion approaches (see, e.g., \cite{Extra2, multipole2, lotan2006}) have been developed (see also the BEM-based approach in \cite{barros2014efficient} and the expansion-based approach of \cite{freed2014perturbative}). Unfortunately, these numerical methods may become prohibitive in terms of computation time for a large number of particles and furthermore, are often formulated in a manner which makes them unsuited for a systematic numerical analysis. The lack of a numerical analysis in turn means that one cannot evaluate \emph{theoretically} the accuracy of these methods, and in particular one cannot explore the dependence of the accuracy on the number of dielectric particles $N$.
	
	The quality of an $N$-body numerical method can be assessed by considering how the accuracy and computational cost of the algorithm scale with $N$. To fix terminology
	
	\begin{itemize}
		\item We say that an $N$-body numerical method is \emph{$N$-error stable} if, for a fixed number of degrees of freedom per object and assuming other properties such as the minimum inter-sphere separation are kept constant,  the {relative} error in the approximate solution does not increase with $N$. Establishing that a numerical method is $N$-error stable requires a rigorous numerical analysis of the algorithm in question, which immediately rules out many of the existing algorithms cited above. 
		
		\item We say that an $N$-body numerical method is \emph{linear scaling in cost} if, for a fixed number of degrees of freedom per object and assuming other properties such as the minimum inter-sphere separation are kept constant,  the numerical method requires $\mathcal{O}(N)$ operations to compute an approximate solution with a given and fixed tolerance. Typically, linear scaling in cost requires the use of fast summation methods such as tree codes (see, e.g., \cite{appel1985efficient,barnes1986hierarchical,li2009cartesian, dehnen2000very, boateng2013comparison, geng2013treecode}) including the so-called Fast Multipole method (see, e.g., \cite{cheng1999fast, greengard1, greengard2, greengard1990numerical}, or particle-mesh and P3M methods (see, e.g., \cite{knebe2001multi, efstathiou1985numerical, hockney}).  If the $N$-body numerical method also involves solving a linear system, then one would additionally have to show that the number of solver iterations required to obtain an approximate solution does not grow with $N$.
		
		\item Finally, we say that an $N$-body numerical method is \emph{linear scaling in accuracy} if it is both $N$-error stable and linear scaling in cost. Linear scaling in accuracy methods can be viewed as the gold-standard for $N$-body problems since these methods require only $\mathcal{O}(N)$ operations to compute an approximate solution with a given average error (the total error scaled by $N$) or relative error.
	\end{itemize}
	
	E. Lindgren and coworkers recently proposed in \cite{lindgren2018}, a computational method based on a Galerkin discretisation of a second-kind boundary integral equation that describes the induced surface charges resulting from the interaction of a large number of dielectric spheres embedded in a homogenous polarisable medium and undergoing mutual polarisation. We emphasise three particular features of this proposed method: 
	
	\begin{enumerate}
		\item[O1)] Numerical experiments suggest that the method is indeed $N$-error stable.
		
		\item[O2)] Numerical experiments indicate that the number of linear solver iterations required to solve the underlying linear system is independent of the number $N$ of dielectric spheres. Since the FMM allows the solution matrix to be multiplied with arbitrary vectors using $\mathcal{O}(N)$ operations, the numerical method also seems to be linear scaling in cost.
		
		\item[O3)] The algorithm is based on a Galerkin discretisation of a second kind boundary integral equation so it is particularly suited for rigorous numerical analysis. 
	\end{enumerate}
	
	O1) and O2) taken together suggest that the proposed Galerkin method is \emph{linear scaling in accuracy}. Importantly however-- and in contrast to several of the existing $N$-body algorithms cited above-- O3) suggests that it might actually be possible to rigorously prove this result. The first article \cite{Hassan} in this series of two papers introduced a new analysis of second kind boundary integral equations posed on spherical domains and presented a detailed error analysis of this method. In particular, the article derived convergence rates for the induced surface charge and total electrostatic energy that did not explicitly depend on the number $N$ of dielectric spherical particles in the system. Consequently, under suitable geometrical assumptions, it was rigorously demonstrated that the numerical method was indeed $N$-error stable.
	
	The goal of the current article is to present a detailed complexity analysis of the Galerkin method proposed by Lindgren and coworkers with the goal of showing that the method is linear scaling in cost. Our main result shows that-- under suitable geometrical assumptions-- the number of linear solver iterations required to solve the linear system obtained from the Galerkin discretisation up to a given tolerance is independent of $N$. Since the FMM allows us to compute approximate matrix-vector products involving the solution matrix using $\mathcal{O}(N)$ operations, it follows that an approximate solution with a given and fixed relative error can indeed be constructed by means of an iterative method using only $\mathcal{O}(N)$ operations. In other words, the numerical method is linear scaling in cost. Combined with the analysis presented in \cite{Hassan} which establishes $N$-error stability, this result confirms that the Galerkin method proposed by E. Lindgren and coworkers is \emph{linear scaling in accuracy}. 
	
	The remainder of this article is organised as follows. In Section \ref{sec:2}, we introduce notation, describe the problem setting and the governing boundary integral equation, and restate key results from our first paper \cite{Hassan}. The complexity analysis, main results and proposed solution strategies are presented in Section \ref{sec:3}. In Section \ref{sec:4}, we present numerical results that support the analysis of Section \ref{sec:3}, and finally in Section \ref{sec:5}, we state our conclusion and discuss possible extensions.
	
	\section{Problem Setting and Previous Results}\label{sec:2}
	Throughout this article, we will use standard results and notation from the theory of boundary integral equations. The setting and notations stated here are essentially identical to those introduced in the first article \cite{Hassan} and are taken primarily from the book of Sauter and Schwab on boundary element methods \cite{Schwab}.
	
	\vspace{2mm}
	
	\subsection{Setting and Notation}\label{sec:2a}~ \vspace{2mm}
	
	{Although we are formally interested in studying all geometrical configurations that are the unions of an arbitrary number $N$ of non-intersecting open balls with varying radii in three dimensions, as pointed out in the first contribution \cite{Hassan}, our claim of $N$-independent bounds requires us to impose certain technical assumptions on the types of geometries we consider. To this end, we denote by $\mathcal{I}$ a countable indexing set, and we consider a so-called family of geometries~$\{\Omega_{\mathcal{F}}\}_{\mathcal{F} \in \mathcal{I}}$. Each element $\Omega_{\mathcal{F}} \subset \mathbb{R}^3$ in this family is the (set) union of a fixed number of non-intersecting open balls of varying locations and radii with associated dielectric constants, and therefore represents a particular physical geometric situation. Consequently, each element $\Omega_{\mathcal{F}}$ of this family of geometries is uniquely characterised by the following four parameters:
		\begin{itemize}
			\item A non-zero number $N_{\mathcal{F}} \in \mathbb{N}$, which represents the total number of dielectric spherical particles that compose the geometry $\Omega_{\mathcal{F}}$;
			\item A collection of points $\{\bold{x}^{\mathcal{F}}_i\}_{i=1}^{N_{\mathcal{F}}} \in \mathbb{R}^3$, which represent the centres of the spherical particles composing the geometry $\Omega_{\mathcal{F}}$;
			\item A collection of positive real numbers $\{r_i^{\mathcal{F}}\}_{i=1}^{N_{\mathcal{F}}} \in \mathbb{R}$, which represent the radii of the spherical particles composing the geometry $\Omega_{\mathcal{F}}$;
			\item A collection of positive real numbers $\{\kappa^{\mathcal{F}}_i\}_{i=0}^{N_{\mathcal{F}}}$. Here, $\kappa^{\mathcal{F}}_0$ denotes the dielectric constant of the external medium while $\{\kappa^{\mathcal{F}}_i\}_{i=1}^{N_{\mathcal{F}}}$ represent the dielectric constants of each dielectric sphere.
		\end{itemize}
		Indeed, using the first three parameters we can define the open balls $\Omega^\mathcal{F}_i:= \mathcal{B}_{r_i}(\bold{x}_i) \subset \mathbb{R}^3$, $i \in \{1, \ldots, N_{\mathcal{F}}\}$ which represent the spherical dielectric particles composing the geometry $\Omega_{\mathcal{F}}$, i.e., $\Omega_{\mathcal{F}}= \cup_{i=1}^{N_\mathcal{F}} \Omega_i^{\mathcal{F}}$. Moreover, the fourth parameter $\{\kappa^{\mathcal{F}}_i\}_{i=0}^N$ denotes the dielectric constants associated with this geometry. Following \cite{Hassan}, we now impose the following assumptions on the above parameters: \vspace{2mm}
		
		\begin{enumerate}
			
			\item[\textbf{A1:}] \textbf{[Uniformly bounded radii]} There exist constants $r^{\infty}_->0$ and $r^{\infty}_+>0$ such that 
			\begin{align*}
				\inf_{\mathcal{F} \in \mathcal{I}}\, \min_{i=1, \ldots, N_{\mathcal{F}}} r^{\mathcal{F}}_i > r^{\infty}_- \quad \text{and} \quad \sup_{\mathcal{F} \in \mathcal{I}}\, \max_{i=1, \ldots, N_{\mathcal{F}}} r^{\mathcal{F}}_i < r^{\infty}_+.
			\end{align*}
			
			\item[\textbf{A2:}] \textbf{[Uniformly bounded minimal separation]} There exists a constant $\epsilon^{\infty} > 0$ such that 
			\begin{align*}
				\inf_{\mathcal{F} \in \mathcal{I}}\, \min_{\substack{i, j=1, \ldots, N_{\mathcal{F} } \\ i \neq j}} \big(\vert \bold{x}_i^{\mathcal{F}} -\bold{x}_j^{\mathcal{F}}\vert - r^{\mathcal{F}}_i -r^{\mathcal{F}}_j\big)> \epsilon^{\infty}.
			\end{align*}
			
			\item[\textbf{A3:}] \textbf{[Uniformly bounded dielectric constants]} There exist constants $\kappa^{\infty}_->0$ and $\kappa^{\infty}_+>0$ such that 
			\begin{align*}
				\inf_{\mathcal{F} \in \mathcal{I}} \,\min_{i=1, \ldots, N_{\mathcal{F}}} \kappa_{i}^{\mathcal{F}}> \kappa^{\infty}_- \quad \text{and} \quad \sup_{\mathcal{F} \in \mathcal{I}} \, \max_{i=1, \ldots, N_{\mathcal{F}}} \kappa_{i}^{\mathcal{F}} < \kappa^{\infty}_+.
			\end{align*}
		\end{enumerate}
		
		In other words we assume that the family of geometries $\{\Omega_{\mathcal{F}}\}_{\mathcal{F} \in \mathcal{I}}$ we consider in this article describe physical situations where the radii of the dielectric spherical particles, the minimum inter-sphere separation distance and the dielectric constants are all uniformly bounded. These assumptions are required because the subsequent bounds we will derive, while explicitly independent of the number of dielectric particles $N_\mathcal{F}$, do depend on other geometrical parameters, and we would thus like to avoid situations where these geometric parameters degrade with increasing $N_\mathcal{F}$. 

		In the remainder of this article, we will consider a fixed geometry from the family of geometries $\{\Omega_{\mathcal{F}}\}_{\mathcal{F} \in \mathcal{I}}$ satisfying the assumptions \textbf{A1)-A3)}. To avoid bulky notation we will drop the superscript and subscript $\mathcal{F}$ and denote this geometry by $\Omega^-$. The geometry is constructed as follows: Let $N \in \mathbb{N}$, let $\{\bold{x}_i\}_{i=1}^N \in \mathbb{R}^3$ be a collection of points in $\mathbb{R}^3$ and let $\{r_i\}_{i=1}^N \in \mathbb{R}$ be a collection of positive real numbers, and for each $i \in \{1, \ldots, N\}$ let $\Omega_i := \mathcal{B}_{r_i}(\bold{x}_i) \subset \mathbb{R}^3$ be the open ball of radius $r_i >0$ centred at the point $\bold{x}_i$. Then $\Omega^- \subset \mathbb{R}^3$ is defined as $\Omega^-:= \cup_{i=1}^N \Omega_i$. Furthermore, we define $\Omega^+:= \mathbb{R}^3 \setminus \overline{\Omega^-}$, and we write $\partial \Omega$ for the boundary of $\Omega^-$ and {$\eta(\bold{x})$ for the unit normal vector at $\bold{x} \in \partial \Omega$ pointing towards the exterior of $\Omega^-$}. Moreover, we denote by $\{\kappa_i\}_{i=1}^N \in \mathbb{R}_+$ the dielectric constants of all spherical particles $\{\Omega_i\}_{i=1}^N$ and by $\kappa_0 \in \mathbb{R}_+$ the dielectric constant of the background medium. For clarity of exposition, we also define the dielectric function $\kappa \colon \partial \Omega \rightarrow \mathbb{R}$ as $\kappa (\bold{x}):= \kappa_i ~\text{ for } \bold{x} \in \partial \Omega_i.$ Notice that by definition for each $i \in \{1, \ldots, N\}$, either $\frac{\kappa-\kappa_0}{\kappa_0}\vert_{\partial \Omega_i} \geq 0$ or $\frac{\kappa-\kappa_0}{\kappa_0}\vert_{\partial \Omega_i} \in (-1, 0]$.} \vspace{2mm}

	Following standard practice, we define the Sobolev space $H^1(\Omega^-):= \left\{u \in L^2(\Omega^-) \colon \nabla u \in L^2(\Omega^-)\right\}$ with the norm $\Vert u \Vert^2_{H^1(\Omega^-)} := \sum_{i=1}^N\Vert u \Vert^2_{L^2(\Omega_i)} + \Vert \nabla u \Vert^2_{L^2(\Omega_i)}$. {Next, we define $C^{\infty}_{\rm comp}(\Omega^+):= \left\{u \vert_{\Omega^+} \colon u \in C_0^{\infty}(\mathbb{R}^3)\right\}$ where $C_0^{\infty}(\mathbb{R}^3)$ denotes the set of infinitely smooth functions with compact support in $\mathbb{R}^3$, and we define the weighted Sobolev space $H^1(\Omega^+)$ as the completion of $C^{\infty}_{\text{comp}}(\Omega^+)$ with respect to the norm $\Vert u \Vert^2_{H^1(\Omega^+)}:= \int_{\Omega^+} \frac{\vert u(\bold{x})\vert^2}{1 + \vert \bold{x}\vert^2}\, d\bold{x}+ \int_{\Omega^+} {\vert \nabla u(\bold{x})\vert^2}\, d\bold{x}$. Functions that satisfy the decay condition associated with exterior Laplace problems will belong to this space\footnote{{The space $H^1(\Omega^+)$ that we have defined here corresponds to the space $H^1(-\Delta, \Omega^+)$  in \cite[Section 2.9.2.4]{Schwab}}. }. In addition, we denote by $H^{\frac{1}{2}}(\partial \Omega)$ the Sobolev space of order $\frac{1}{2}$ equipped with the Sobolev-Slobodeckij norm $\Vert \lambda \Vert^2_{H^{\frac{1}{2}}(\partial \Omega)}:=\sum_{i=1}^N \Vert \lambda\Vert^2_{L^2(\partial \Omega_i)} + \int_{\partial \Omega_i} \int_{\partial \Omega_i} \frac{\vert\lambda(\bold{x})-\lambda(\bold{y})\vert^2}{\vert \bold{x} - \bold{y}\vert^3 } \, ds_{\bold{x}} ds_{\bold{y}}.$ Moreover, we define $H^{-\frac{1}{2}}(\partial \Omega):=\left(H^{\frac{1}{2}}(\partial \Omega)\right)^*$ and we equip this Sobolev space with the canonical dual norm 
		\begin{align*}
			\Vert \sigma \Vert_{H^{-\frac{1}{2}}(\partial \Omega)}:= \sup_{0 \neq \psi \in {H}^{\frac{1}{2}}(\partial \Omega)} \frac{\langle \sigma, \psi \rangle_{H^{-\frac{1}{2}}(\partial \Omega) \times H^{\frac{1}{2}}(\partial \Omega)}}{\Vert \psi \Vert_{H^{\frac{1}{2}} (\partial \Omega)}} \qquad \forall \sigma \in H^{-\frac{1}{2}}(\partial \Omega),
		\end{align*}
		where $\langle \cdot, \cdot \rangle_{H^{-\frac{1}{2}}(\partial \Omega) \times H^{\frac{1}{2}}(\partial \Omega)}$ denotes the duality pairing between $H^{-\frac{1}{2}}(\partial \Omega)$ and $H^{\frac{1}{2}}(\partial \Omega)$}. \vspace{2mm}
	
	We introduce $\gamma^- \colon H^1(\Omega^{-}) \rightarrow H^{\frac{1}{2}}(\partial \Omega)$ and $\gamma^+ \colon H^1(\Omega^+) \rightarrow H^{\frac{1}{2}}(\partial \Omega)$ as the continuous, linear and surjective interior and exterior Dirichlet trace operators respectively (see, e.g., \cite[Theorem 2.6.8, Theorem 2.6.11]{Schwab}). Moreover, for $s \in \{+, -\}$, we define the closed subspace $\mathbb{H}(\Omega^s):= 
	\{u \in H^{1}(\Omega^s) \colon \Delta u =0 \text{ in } \Omega^s\},$
	and we write $\gamma^-_N \colon \mathbb{H}(\Omega^-) \rightarrow H^{-\frac{1}{2}}(\partial \Omega)$ and $\gamma^+_N \colon \mathbb{H}(\Omega^+) \rightarrow H^{-\frac{1}{2}}(\partial \Omega)$ for the interior and exterior Neumann trace operator respectively (see \cite[Theorem 2.8.3]{Schwab} for precise conventions). The interior and exterior Dirichlet and Neumann trace operators can be defined analogously for functions of appropriate regularity on $\Omega^- \cup \Omega^+$ or $\mathbb{R}^3$. Furthermore, we introduce $\text{DtN} \colon H^{\frac{1}{2}}(\partial \Omega) \rightarrow H^{-\frac{1}{2}}(\partial \Omega)$ as the so-called (interior) Dirichlet-to-Neumann map that takes as input boundary data in $H^{\frac{1}{2}}(\partial \Omega)$, computes the interior harmonic extension in $H^1(\Omega^-)$, and yields as output the interior Neumann trace of the harmonic extension in $H^{-\frac{1}{2}}(\partial \Omega)$. Note that Local Dirichlet-to-Neumann maps may be defined analogously on each sphere $\partial \Omega_i, ~i=1, \ldots, N$.\vspace{2mm}
	
	Next, for each $\nu \in H^{-\frac{1}{2}}(\partial \Omega),~ \lambda \in H^{\frac{1}{2}}(\partial \Omega)$ and all $\bold{x} \in \mathbb{R}^3 \setminus \partial \Omega$ we define the functions
	\begin{align*}
		\mathcal{S}(\nu)(\bold{x}):=\int_{\partial \Omega} \frac{\nu(\bold{y})}{4\pi\vert \bold{x}- \bold{y}\vert}\, d \bold{y}, \qquad \text{and} \qquad
		\mathcal{D}(\lambda)(\bold{x}):=\int_{\partial \Omega} \lambda(\bold{y}) \eta(\bold{y}) \cdot\nabla_{\bold{y}}\frac{1}{4\pi\vert \bold{x}- \bold{y}\vert}\, d \bold{y}.
	\end{align*}
	The mappings $\mathcal{S}$ and $\mathcal{D}$ are the single layer and double layer potentials respectively. It can be shown (see, e.g., \cite[Chapter 2]{Schwab}) that $\mathcal{S}$ is a linear bounded operator from $H^{-\frac{1}{2}}(\partial \Omega)$ to $H^{1}_{\rm loc}\left(\mathbb{R}^3\right)$ and $\mathcal{D}$ is a linear bounded operator from $H^{\frac{1}{2}}(\partial \Omega)$ to $H^{1}_{\text{loc}}\left(\mathbb{R}^3 \setminus \partial \Omega\right)$, and both $\mathcal{S}$ and $\mathcal{D}$ map into the space of harmonic functions on the complement $\mathbb{R}^3 \setminus \partial \Omega$ of the boundary. Equipped with these potentials, we define the following bounded linear boundary integral operators:
	\begin{align*}
		\mathcal{V}&:= \hphantom{-}\big(\gamma^- \circ \mathcal{S}\big) \hspace{0.0cm}\colon \hspace{0.0cm}H^{-\frac{1}{2}}(\partial \Omega) \rightarrow H^{\frac{1}{2}}(\partial \Omega), \hspace{1.0cm} \mathcal{K}^{\hphantom{*}}:= \Big(\gamma^- \circ\mathcal{D} + \frac{1}{2}I\Big)\colon H^{\frac{1}{2}}(\partial \Omega) \rightarrow H^{\frac{1}{2}}(\partial \Omega),\\
		\mathcal{W}&:= -\big(\gamma_N^- \circ\mathcal{D}\big) \hspace{0.0cm}\colon \hspace{0.0cm}H^{\frac{1}{2}}(\partial \Omega) \rightarrow H^{-\frac{1}{2}}(\partial \Omega),
		\hspace{1cm}
		\mathcal{K}^*:= \Big(\gamma_N^- \circ\mathcal{S} - \frac{1}{2}I\Big)\colon H^{-\frac{1}{2}}(\partial \Omega) \rightarrow H^{-\frac{1}{2}}(\partial \Omega).
	\end{align*}
	
	Here $I$ denotes the identity operator on the relevant trace space. The mapping $\mathcal{V}$ is the single layer boundary operator, the mapping $\mathcal{K}$ is the double layer boundary operator, the mapping $\mathcal{K}^*$ is the adjoint double layer boundary operator and the mapping $\mathcal{W}$ is the hypersingular boundary operator. Detailed definitions and properties of these boundary integral operators can be found in \cite[Chapter 3]{Schwab}. We state two properties in particular that will be used in the sequel. \vspace{2mm}
	
	{{\textbf{Property 1:} \cite[Theorem 3.5.3]{Schwab} The single layer boundary operator $\mathcal{V} \colon H^{-\frac{1}{2}}(\partial \Omega) \rightarrow H^{\frac{1}{2}}(\partial \Omega)$ is hermitian and coercive, i.e., there exists a constant $c_{\mathcal{V}} > 0$ such that for all $\sigma \in H^{-\frac{1}{2}}(\partial \Omega)$ it holds that}
		\begin{align*}
			\langle \sigma, \mathcal{V}\sigma \rangle_{H^{-\frac{1}{2}}(\partial \Omega) \times H^{\frac{1}{2}}(\partial \Omega)} \geq c_{\mathcal{V}} \Vert \sigma\Vert^2_{H^{-\frac{1}{2}}(\partial \Omega)}.
		\end{align*}
		
		This implies in particular that the inverse $\mathcal{V}^{-1} \colon H^{\frac{1}{2}}(\partial \Omega) \rightarrow H^{-\frac{1}{2}}(\partial \Omega)$ is also a hermitian, coercive and bounded linear operator. Consequently, $\mathcal{V}$ induces a norm $\Vert \cdot \Vert_{\mathcal{V}}$ and associated inner product on $H^{-\frac{1}{2}}(\partial \Omega)$ and the inverse $\mathcal{V}^{-1}$ induces a norm $\Vert \cdot \Vert_{\mathcal{V}^{-1}}$ and associated inner product on $H^{\frac{1}{2}}(\partial \Omega)$.}\\
	
	{	{\textbf{Property 2:}\cite[Theorem 3.5.3]{Schwab}} The hypersingular boundary operator $\mathcal{W} \colon H^{\frac{1}{2}}(\partial \Omega) \rightarrow H^{\frac{1}{2}}(\partial \Omega)$ is hermitian, non-negative and coercive on a subspace of $H^{\frac{1}{2}}(\partial \Omega)$, i.e., there exists a constant $c_{\mathcal{W}} > 0$ such that for all functions $\lambda \in H^{\frac{1}{2}}(\partial \Omega)$ with $\sum_{i=1}^N\left\vert\int_{\partial \Omega_i} \lambda(\bold{x})\, d \bold{x}\right\vert=0$, it holds that
		\begin{align*}
			\langle \mathcal{W}\lambda, \lambda \rangle_{H^{-\frac{1}{2}}(\partial \Omega) \times H^{\frac{1}{2}}(\partial \Omega)} \geq c_{\mathcal{W}} \Vert \lambda\Vert^2_{H^{\frac{1}{2}}(\partial \Omega)}.
		\end{align*}	
		
\vspace{2mm}
		\subsection{Dielectric Spheres Electrostatic Interaction Problem}~\vspace{2mm}
		
		Let us now state the problem we wish to analyse. In order to avoid trivial situations, we assume throughout this article that $\kappa_j \neq \kappa_0$ for all $j=1, \ldots, N$ (see also the justification of this assumption in Remark 2.5 in \cite{Hassan}).\vspace{2mm}
		
		\noindent{\textbf{Integral Equation Formulation for the Induced Charges}}~
		
		Let $\sigma_f \in H^{-\frac{1}{2}}(\partial \Omega)$. Find $\nu \in H^{-\frac{1}{2}}(\partial \Omega)$ with the property that
		\begin{align}\label{eq:3.3a}
			\nu - \frac{\kappa_0-\kappa}{\kappa_0} (\text{DtN}\mathcal{V})\nu= \frac{4\pi}{\kappa_0}\sigma_f.
		\end{align}
	
		Here, $\sigma_f \in H^{-\frac{1}{2}}(\partial \Omega)$ is called the free charge and is a known quantity. Physically, this is the charge (up to a scaling factor) on each dielectric sphere in the absence of any polarisation effects, i.e., if $\kappa=\kappa_0$. The unknown $\nu \in H^{-\frac{1}{2}}(\partial \Omega)$ is called the induced surface charge. Physically, this is the charge distribution that results on each dielectric sphere after including polarisation effects.

		\begin{remark}
			Details on how to derive the boundary integral equation \eqref{eq:3.3a} from a PDE-based transmission problem and a proof of our initial claim that the BIE \eqref{eq:3.3a} is a boundary integral equation of the second kind can be found in \cite{Hassan}. 
		\end{remark}

		As discussed in the first paper \cite{Hassan}, a direct analysis of the BIE \eqref{eq:3.3a} is not feasible if we wish to obtain continuity and inf-sup constants that are independent of the number $N$ of dielectric particles in our problem. Consequently, it is necessary to adopt an indirect approach and reformulate the boundary integral equation~\eqref{eq:3.3a} in terms of a so-called surface electrostatic potential $\lambda:= \mathcal{V}\nu \in H^{\frac{1}{2}}(\partial \Omega)$. {Indeed, in view of Property 1 above we see that $\lambda =\mathcal{V}\nu$ defines an isomorphism from $H^{-\frac{1}{2}}(\partial \Omega)$ to $H^{\frac{1}{2}}(\partial \Omega)$ and we can thus write an equivalent~BIE.}\vspace{3mm}
		
		\noindent{\textbf{Integral Equation Formulation for the Surface Electrostatic Potential}}~
		
		Let $\sigma_f \in H^{-\frac{1}{2}}(\partial \Omega)$. Find $\lambda \in H^{\frac{1}{2}}(\partial \Omega)$ with the property that
		\begin{align}\label{eq:3.3}
			\lambda - \mathcal{V} \text{DtN}\Big(\frac{\kappa_0-\kappa}{\kappa_0} \lambda\Big)= \frac{4\pi}{\kappa_0}\mathcal{V}\sigma_f.
		\end{align}
		
		For clarity of exposition, we define the relevant boundary integral operators.
		\begin{definition}\label{def:A}
			We define the linear operator $\mathcal{A} \colon {H}^{\frac{1}{2}}(\partial \Omega) \rightarrow {H}^{\frac{1}{2}}(\partial \Omega)$ as
			\begin{align*}
				\mathcal{A} \lambda:= \lambda - \mathcal{V} \text{DtN}\Big(\frac{\kappa_0-\kappa}{\kappa_0} \lambda\Big) \qquad \forall \lambda \in {H}^{\frac{1}{2}}(\partial \Omega).
			\end{align*}
			In addition, we denote by $ \mathcal{A}^* \colon {H}^{-\frac{1}{2}}(\partial \Omega) \rightarrow {H}^{-\frac{1}{2}}(\partial \Omega)$ the adjoint operator of $\mathcal{A}$. 
		\end{definition}

		Next, we define the approximation spaces and state the Galerkin discretisation of the boundary integral equation \eqref{eq:3.3a}. In the sequel, we will denote by $\mathbb{N}_0$ the set of non-negative integers.\vspace{2mm}
		
		\noindent \textbf{Notation:} Let $\ell \in \mathbb{N}_0$ and $m \in \{-\ell, \ldots, \ell\}$ be integers. We denote by $\mathcal{Y}_{\ell m} \colon \mathbb{S}^2 \rightarrow \mathbb{R}$ the real-valued $L^2$-orthonormal spherical harmonic of degree $\ell$ and order $m$. 

		\begin{definition}[Approximation Space on a Sphere]\label{def:6.6}~
			Let $\mathcal{O}_{\bold{x}_0} \subset \mathbb{R}^3$ be an open ball of radius $r > 0$ centred at the point $\bold{x}_0 \in \mathbb{R}^3$ and let $\ell_{\max} \in \mathbb{N}_0$. We define the finite-dimensional Hilbert space $W^{\ell_{\max}}(\partial \mathcal{O}_{\bold{x}_0}) \subset {H}^{\frac{1}{2}}(\partial\mathcal{O}_{\bold{x}_0}) \subset  {H}^{-\frac{1}{2}}(\partial\mathcal{O}_{\bold{x}_0})$ as the vector space
			\begin{align*}
				W^{\ell_{\max}}(\partial \mathcal{O}_{\bold{x}_0}):= \Big\{u \colon \partial \mathcal{O}_{\bold{x}_0} \rightarrow \mathbb{R} \text{ such that } u(\bold{x})= \sum_{{\ell}=0}^{\ell_{\max}} \sum_{m=-\ell}^{m=+\ell} [u]^\ell_m \mathcal{Y}_{\ell m}\left(\frac{\bold{x}-\bold{x}_0}{\vert \bold{x}-\bold{x}_0\vert}\right) \text{where all }[u]_{\ell}^m \in \mathbb{R}\Big\},
			\end{align*}
			equipped with the inner product
			\begin{align}\label{eq:SH1}
				(u, v)_{W^{\ell_{\max}}(\partial \mathcal{O}_{\bold{x}_0})}:=r^2 [u]_0^0 [v]_0^0 + r^2 \sum_{\ell=1}^{\ell_{\max}} \sum_{m=-\ell}^{m=+\ell} \frac{\ell}{r} [u]_{\ell}^m [v]_{\ell}^m \qquad \forall u, v \in W^{\ell_{\max}}(\partial \mathcal{O}_{\bold{x}_0}).
			\end{align}	
		\end{definition}
		
		It is now straightforward to extend the Hilbert space defined in Definition \ref{def:6.6} to the domain $\partial \Omega$.\vspace{3mm}
		
		\begin{definition}[Global Approximation Space]\label{def:6.7}~
			Let $\ell_{\max} \in \mathbb{N}_0$. We define the finite-dimensional Hilbert space $W^{\ell_{\max}}(\partial \Omega) \subset H^{\frac{1}{2}}(\partial \Omega) \subset H^{-\frac{1}{2}}(\partial \Omega)$ as the vector space
			\begin{align*}
				W^{\ell_{\max}}(\partial \Omega) := \Big\{u \colon \partial\Omega \rightarrow \mathbb{R} \text{ such that } \forall i \in \{1, \ldots, N\} \colon u\vert_{\partial \Omega_i} \in W^{\ell_{\max}}(\partial \Omega_i)\Big\},
			\end{align*}
			equipped with the inner product	$(u, v)_{W^{\ell_{\max}}(\partial \Omega) }:= \sum_{i=1}^N \left(u, v\right)_{W^{\ell_{\max}}(\partial \Omega_i)} ~\forall u, v \in W^{\ell_{\max}}(\partial \Omega) $.
			
		\end{definition}
		
		\vspace{2mm}
		
		\noindent{\textbf{Galerkin Discretisation of the Integral Equation \eqref{eq:3.3a}}}~
		
		Let $\sigma_f \in {H}^{-\frac{1}{2}}(\partial \Omega)$ and let $\ell_{\max} \in \mathbb{N}$. Find $\nu_{\ell_{\max}} \in W^{\ell_{\max}}(\partial \Omega) $ such that for all $\psi_{\ell_{\max}} \in W^{\ell_{\max}}(\partial \Omega) $ it holds that
		\begin{align}\label{eq:Galerkina}
			\left(\mathcal{A}^*\nu_{\ell_{\max}}, \psi_{\ell_{\max}}\right)_{L^2(\partial \Omega)}= \frac{4\pi}{\kappa_0}\left(\sigma_f, \psi_{\ell_{\max}}\right)_{L^2(\partial \Omega)}.
		\end{align}
		
		Once again, a direct analysis of the Galerkin discretisation \eqref{eq:Galerkina} is not feasible since we are unable to obtain $N$-independent stability constants. This difficulty was circumvented in \cite{Hassan} through the introduction of a so-called ``reduced'' global approximation space and a ``reduced'' Galerkin discretisation of the BIE \eqref{eq:3.3} for the surface electrostatic potential. We present here only the essentials of this approach.
		
		\begin{definition}\label{def:decomp}
			We define the $N$-dimensional, closed subspace $\mathcal{C}(\partial \Omega) \subset H^{\frac{1}{2}}(\partial \Omega)$ as the vector space
			\begin{align*}
				\mathcal{C}(\partial \Omega):= \left\{u \colon \partial \Omega \rightarrow \mathbb{R} \colon \forall i =1, \ldots, N \text{ the restriction }u|_{\partial \Omega_i} \text{ is a constant function}\right\},
			\end{align*}
			equipped with the $L^2(\partial \Omega)$ norm. Additionally, we define the subspaces $\breve{H}^{\frac{1}{2}}(\partial \Omega) \subset {H}^{\frac{1}{2}}(\partial \Omega)$ and $\breve{H}^{-\frac{1}{2}}(\partial \Omega) \subset {H}^{-\frac{1}{2}}(\partial \Omega)$ as the vector spaces
			\begin{align*}
				\breve{H}^{\frac{1}{2}}(\partial \Omega)&:=\left\{u \in H^{\frac{1}{2}}(\partial \Omega) \colon (u, v)_{L^2(\partial \Omega)}=0 \hspace{21mm} \forall v \in \mathcal{C}(\partial \Omega)\right\},\\
				\breve{H}^{-\frac{1}{2}}(\partial \Omega)&:=\left\{\phi \in H^{-\frac{1}{2}}(\partial \Omega) \colon \langle \phi, v\rangle_{H^{-\frac{1}{2}}(\partial \Omega) \times H^{\frac{1}{2}}(\partial \Omega)}=0 \hspace{3mm}\forall v \in \mathcal{C}(\partial \Omega)\right\},
			\end{align*}
			equipped with the respective fractional Sobolev norms introduced earlier. Moreover, we denote the projection operators associated with these decompositions as $\mathbb{P}_{0} \colon H^{\frac{1}{2}}(\partial \Omega) \rightarrow  \mathcal{C}(\partial \Omega)$, $\mathbb{Q}_{0}\colon H^{-\frac{1}{2}}(\partial \Omega) \rightarrow \mathcal{C}(\Omega)$, $\mathbb{P}_{0}^{\perp} \colon H^{\frac{1}{2}}(\partial \Omega) \rightarrow \breve{H}^{\frac{1}{2}}(\partial \Omega)$, and $\mathbb{Q}_{0}^{\perp} \colon H^{\frac{1}{2}}(\partial \Omega) \rightarrow \breve{H}^{-\frac{1}{2}}(\partial \Omega)$.
		\end{definition}
		

		Intuitively, the spaces $\breve{H}^{\frac{1}{2}}(\partial \Omega)$ and $\breve{H}^{-\frac{1}{2}}(\partial \Omega)$ are trace spaces that do not contain any piecewise constant functions. It can therefore be shown that on these spaces, the Dirichlet-to-Neumann map $\text{DtN}\colon \breve{H}^{\frac{1}{2}}(\partial \Omega) \rightarrow \breve{H}^{-\frac{1}{2}}(\partial \Omega)$ is a continuous bijection. Naturally, these spaces can also be defined on an individual sphere. \vspace{2mm}
		
		Next, we introduce new norms on the underlying trace spaces. As shown in \cite{Hassan}, these norms are an essential ingredient in the analysis of the integral operators $\mathcal{A}^*$ and $\mathcal{A}$.
		
		\begin{definition}\label{def:NewNorm}
			We define on $H^{\frac{1}{2}}(\partial \Omega)$ a new norm $||| \cdot ||| \colon H^{\frac{1}{2}}(\partial \Omega) \rightarrow \mathbb{R}$ given by
			\begin{align*}
				\forall \lambda \in H^{\frac{1}{2}}(\partial \Omega)\colon ~ |||\lambda|||^2:= \left\Vert \mathbb{P}_0 \lambda\right\Vert^2_{L^2(\partial \Omega)}+ \left\langle \text{DtN}\lambda, \lambda\right \rangle_{H^{-\frac{1}{2}}(\partial \Omega) \times H^{\frac{1}{2}}(\partial \Omega)},
			\end{align*}
			and we define on $H^{-\frac{1}{2}}(\partial \Omega)$ a new dual norm $||| \cdot |||^* \colon H^{-\frac{1}{2}}(\partial \Omega) \rightarrow \mathbb{R}$ given by
			\begin{align*}
				||| \sigma |||^*:= \sup_{0\neq \psi \in H^{\frac{1}{2}}(\partial \Omega) } \frac{\left \langle \sigma, \psi \right \rangle_{H^{-\frac{1}{2}}(\partial \Omega) \times H^{\frac{1}{2}}(\partial \Omega)}}{||| \psi |||}.
			\end{align*} 
		\end{definition}
		
		It was shown in \cite{Hassan} that the norm $||| \cdot |||$ is equivalent to the usual $\Vert \cdot \Vert_{H^{\frac{1}{2}}(\partial \Omega)}$ norm, i.e., there exists a constant $c_{\rm equiv} >1$ that is independent of $N$ such that for all $\lambda \in H^{\frac{1}{2}}(\partial \Omega)$ it holds that $\frac{1}{c_{\rm equiv}} ||| \lambda ||| \leq \Vert  \lambda \Vert_{H^{\frac{1}{2}}(\partial \Omega)} \leq c_{\rm equiv} ||| \lambda |||$. Similarly, the new $||| \cdot |||^*$ dual norm on ${H}^{-\frac{1}{2}}(\partial \Omega)$ is equivalent to the canonical dual norm $\Vert \cdot \Vert_{H^{-\frac{1}{2}}(\partial \Omega)}$ with an equivalence constant that is once again independent of $N$. Furthermore, it is a simple exercise to prove that for all $\tilde{\lambda} \in \breve{H}^{\frac{1}{2}}(\partial \Omega)$ it holds that
		\begin{align*}
			||| \text{DtN} \tilde{\lambda} |||^* = ||| \tilde{\lambda} |||.
		\end{align*}
		
		In the sequel, we adopt the convention that the Hilbert spaces ${H}^{\frac{1}{2}}(\partial \Omega), \breve{H}^{\frac{1}{2}}(\partial \Omega)$ are equipped with the $||| \cdot |||$ norm (and associated inner product) and that the dual spaces $H^{-\frac{1}{2}}(\partial \Omega), \breve{H}^{-\frac{1}{2}}(\partial \Omega)$ are equipped with the $||| \cdot |||^*$ norm. Notice that on the space $W^{\ell_{\max}}(\partial \Omega) $, the $||| \cdot |||$ norm and the $\Vert \cdot \Vert_{W^{\ell_{\max}}(\partial \Omega) }$ norm coincide. With this convention, we define additional mappings that will be of use in the next section.
		
		\begin{definition}[Orthogonal Projectors on the Approximation Space]\label{def:PQ}~
			Let $\ell_{\max} \in \mathbb{N}_0$ and let the approximation space $W^{\ell_{\max}}(\partial \Omega) $ be defined as in Definition \ref{def:6.7}. We denote by $\mathbb{P}_{\ell_{\max}} \colon H^{\frac{1}{2}}(\partial\Omega) \rightarrow W^{\ell_{\max}}(\partial \Omega) $ and $\mathbb{Q}_{\ell_{\max}}\colon H^{-\frac{1}{2}}(\partial\Omega) \rightarrow W^{\ell_{\max}}(\partial \Omega) $ the orthogonal projection operators on these spaces.
		\end{definition}
		
		
		
		\begin{definition}\label{def:tildeA}
			We define the linear operator $\widetilde{\mathcal{A}} \colon \breve{H}^{\frac{1}{2}}(\partial \Omega) \rightarrow \breve{H}^{\frac{1}{2}}(\partial \Omega)$ as $\widetilde{\mathcal{A}}:= \mathbb{P}^{\perp}_0 \mathcal{A} \mathbb{P}_0^{\perp}$, and we refer to $\widetilde{\mathcal{A} }$ as the `modified' boundary integral operator.
		\end{definition}

		Finally, we define the ``reduced'' approximation spaces associated with the Galerkin discretisation of the `modified' boundary integral operator $\widetilde{\mathcal{A} }$. 
		
		\begin{definition}[Reduced Global Approximation Space]\label{def:Appromxation2}
			Let $\ell_{\max} \in \mathbb{N}$. We define constructively the finite-dimensional Hilbert space $W_0^{\ell_{\max}}(\partial \Omega) \subset \breve{H}^{\frac{1}{2}}(\partial \Omega)$ as the set
			\begin{align*}
				W_0^{\ell_{\max}}(\partial \Omega):= \Big\{u \in W^{\ell_{\max}}(\partial \Omega)\colon \mathbb{P}_0 u = 0\Big\},
			\end{align*}
			equipped with the $(\cdot, \cdot)_{W^{\ell_{\max}}(\partial \Omega) }$ inner product.
		\end{definition}

		\noindent{\textbf{``Reduced'' Galerkin Discretisation of the Integral Equation \eqref{eq:3.3} with Modified RHS} }~
		
		Let $\sigma_f \in {H}^{-\frac{1}{2}}(\partial \Omega)$, let $\ell_{\max} \in \mathbb{N}$, and let $\mathbb{Q}_{\ell_{\max}} \colon H^{-\frac{1}{2}}(\partial \Omega) \rightarrow W^{\ell_{\max}}(\partial \Omega) $ be the orthogonal projection operator. Find $\lambda_{\ell_{\max}} \in W_0^{\ell_{\max}}(\partial \Omega)$ such that for all $\sigma_{\ell_{\max}} \in W_0^{\ell_{\max}}(\partial \Omega)$ it holds that
		\begin{align}\label{eq:Galerkin}
			\left(\widetilde{\mathcal{A}}\lambda_{\ell_{\max}}, \sigma_{\ell_{\max}}\right)_{L^2(\partial \Omega)}= \frac{4\pi}{\kappa_0}\left(\mathcal{V}\mathbb{Q}_{\ell_{\max}}\sigma_f, \sigma_{\ell_{\max}}\right)_{L^2(\partial \Omega)}.
		\end{align}
		
		Notice that Equation \eqref{eq:Galerkin} is not exactly the Galerkin discretisation of the BIE \eqref{eq:3.3} on the reduced approximation space $W_0^{\ell_{\max}}(\partial \Omega)$. This is because the right-hand side of Equation \eqref{eq:Galerkin} is not precisely the projection of the right-hand side of the BIE \eqref{eq:3.3} onto $W_0^{\ell_{\max}}(\partial \Omega)$ unless $\sigma_f \in W^{\ell_{\max}}(\partial \Omega) $. There are two main reasons we introduce this so-called ``reduced'' Galerkin discretisation of BIE \eqref{eq:3.3} with modified right-hand side. First, it is possible to obtain continuity and discrete stability constants for $\widetilde{\mathcal{A} }$ that are \emph{independent of $N$}. Indeed, we have the following result:
		\begin{theorem}\label{lem:well-posed}
			{Let the set $N_{+} \subset \mathbb{N}$ consist of indices $i \in \{1, \ldots, N\}$ such that $\kappa\vert_{\partial \Omega_i} > \kappa_0$ and the set $N_{-} \subset \mathbb{N}$ consist of indices $j \in \{1, \ldots, N\}$ such that $\kappa\vert_{\partial \Omega_j} < \kappa_0$, and let the constants $C_{\tilde{\mathcal{A}}}$ and $\beta_{\tilde{\mathcal{A}}}$ be defined as}
			\begin{align}\label{eq:contin}
				C_{\tilde{\mathcal{A}}}:=1+\max\left \vert \frac{\kappa - \kappa_0}{\kappa_0}\right \vert\cdot \left(\frac{c_{\rm equiv}}{\sqrt{c_{\mathcal{V}}}} \right), \qquad \text{and} \qquad \beta_{\tilde{\mathcal{A}}}:= \frac{\min\left\{ \min_{j \in N_+}\frac{\kappa_j-\kappa_0}{\kappa_0},~ \min_{j \in N_-} \frac{\kappa_j}{\kappa_0}\frac{\kappa_0-\kappa_j}{\kappa_0} \right\}}{ \max_{j=1, \ldots, N}  \Big \vert \frac{\kappa_j-\kappa_0}{\kappa_0}\Big \vert}.
			\end{align}
			Then it holds that 
			
			\begin{enumerate}
				\item The operator $\widetilde{\mathcal{A}} \colon \breve{H}^{\frac{1}{2}}(\partial \Omega) \rightarrow \breve{H}^{\frac{1}{2}}(\partial \Omega)$ has dense range, is bounded above with constant $C_{\tilde{\mathcal{A}}}$ and bounded below with constant~$\beta_{\tilde{\mathcal{A}}}$.
				\item The finite-dimensional operator $\mathbb{P}_{\ell_{\max}}\widetilde{\mathcal{A}} \colon W_0^{\ell_{\max}}(\partial \Omega) \rightarrow W_0^{\ell_{\max}}(\partial \Omega)$ is also bounded above with constant $C_{\tilde{\mathcal{A}}}$ and bounded below with constant $\beta_{\tilde{\mathcal{A}}}$.
			\end{enumerate}
		\end{theorem}
		\begin{proof}
			See Lemma 4.5, 4.10, and 4.13 in \cite{Hassan}.
		\end{proof}

		Second, as the next lemma shows, if we have obtained a solution to Equation \eqref{eq:Galerkin} then a solution to the original Galerkin discretisation \eqref{eq:Galerkina} for the induced surface charge can be computed easily.
		\begin{lemma}\label{lem:indirect}
			Let $\sigma_f \in {H}^{-\frac{1}{2}}(\partial \Omega)$, let $\ell_{\max} \in \mathbb{N}$, let $\mathbb{Q}_{\ell_{\max}} \colon H^{-\frac{1}{2}}(\partial \Omega) \rightarrow W^{\ell_{\max}}(\partial \Omega) $ denote the orthogonal projection operator, let $\lambda_{\ell_{\max}} \in W^{\ell_{\max}}_0(\partial \Omega) $ be the solution of  the ``reduced'' Galerkin discretisation \eqref{eq:Galerkin} with right hand side defined through $\sigma_f$, and let $\nu_{\ell_{\max}} \in W^{\ell_{\max}}(\partial \Omega) $ be the solution of the original Galerkin discretisation \eqref{eq:Galerkina} with right hand side defined through $\sigma_f$. Then it holds that
			\begin{align*}
				\nu_{\ell_{\max}} = \frac{\kappa_0-\kappa}{\kappa_0}\text{\rm DtN}\lambda_{\ell_{\max}} +\frac{4\pi}{\kappa_0} \mathbb{Q}_{\ell_{\max}}\sigma_f.
			\end{align*}
		\end{lemma}
		\begin{proof}
			See the proof of Theorem 4.17 in \cite{Hassan}.
		\end{proof}
		
		
		Theorem \ref{lem:well-posed} and Lemma \ref{lem:indirect} can together be used to prove that both the BIEs \eqref{eq:3.3a} and \eqref{eq:3.3} as well as the Galerkin discretisations \eqref{eq:Galerkina} and \eqref{eq:Galerkin} are well-posed, and to obtain convergence rates for the induced surface charge that do not explicitly depend on $N$ (see \cite{Hassan} for details).
		
		\section{Convergence Analysis of the Linear Solver and Solution Strategies}\label{sec:3}
		
		Throughout this section, we assume the setting described in Section \ref{sec:2}. Before proceeding to our main results however, let us describe in more detail the obstacles we face in performing a detailed convergence analysis. 
		
		Our main goal it to prove that we can obtain a solution-- up to given and fixed tolerance-- to the Galerkin discretisation \eqref{eq:Galerkina} for the approximate induced surface charge $\nu_{\ell_{\max}}$ \emph{using only $\mathcal{O}(N)$ operations.} Naturally, computing a solution to the finite-dimensional equation \eqref{eq:Galerkina} involves solving a linear system, which is typically done using a Krylov subspace solver. Consequently, in order to obtain the required $\mathcal{O}(N)$ scaling, we must show~that
		\begin{itemize}
			\item[B1)] It is possible to perform matrix vector multiplications involving the underlying solution matrix using $\mathcal{O}(N)$ operations;
			\item[B2)] It is possible to obtain an upper bound \underline{independent of $N$} on the number of iterations of the Krylov subspace solver required to obtain the solution of the linear system (up to some tolerance).
		\end{itemize}
		
		Let us first consider B1). Notice that the underlying boundary integral operators $\mathcal{A}^*$ and $\mathcal{A}$ defined through Definition \ref{def:A} are constructed from the single layer boundary operator $\mathcal{V}$, the Dirichlet-to-Neumann map $\text{DtN}$, the identity map, and the dielectric function $\kappa$. Given our choice of approximation space $W^{\ell_{\max}}$, the $\text{DtN}$ map, the identity map and the dielectric function $\kappa$ can be written as diagonal matrices. Furthermore, the Fast Multipole method (FMM) can be used to compute the action of $\mathcal{V}$ on an arbitrary element of $W^{\ell_{\max}}$ up to a given accuracy in $\mathcal{O}(N)$ operations. Thus, all matrix-vector products involving the solution matrix obtained from the discretisation of $\mathcal{A}^*$ (and also $\mathcal{A}$) can be performed in $\mathcal{O}(N)$ up to a given accuracy (see also \cite{lindgren2018} for more details).

		Next, let us consider the statement B2), which requires us to perform a convergence analysis of the Krylov subspace solver used to solve the linear system arising from Galerkin discretisation \eqref{eq:Galerkina}. Here, we encounter the first obstacle: It is well known that the convergence behaviour of Krylov subspace solvers depends crucially on the spectrum of the operator that is being discretised, i.e., $\mathcal{A}^*$. The spectrum in turn depends on the continuity and coercivity/inf-sup constants of the operator, which themselves are obtained through a detailed numerical analysis. However, the analysis we presented in the first paper \cite{Hassan} (see also Theorem \ref{lem:well-posed}) derived explicit $N$-independent continuity and inf-sup constants only for the \emph{`modified' boundary integral operator} $\widetilde{\mathcal{A} }$ and used an indirect approach to obtain convergence rates for the Galerkin discretisation \eqref{eq:Galerkina}. Unfortunately, we have little explicit information on the spectrum of $\mathcal{A}^*$ and certainly no $N$-independent bounds on the largest and smallest eigenvalues.

		Our solution to this obstacle is to adopt once again an indirect route: Given a free charge $\sigma_f \in H^{-\frac{1}{2}}(\partial \Omega)$
		
		\begin{itemize}
			\item[Step 1)] We solve the `reduced' Galerkin discretisation \eqref{eq:Galerkin} to obtain $\lambda_{\ell_{\max}} \in W_0^{\ell_{\max}}(\partial \Omega)$ for some $\ell_{\max} \in \mathbb{N}$. This step will require the solution of a linear system using a Krylov subspace solver.
			
			\item[Step 2)] We make use of Lemma \ref{lem:indirect} to compute the solution $\nu_{\ell_{\max}} \in W^{\ell_{\max}}(\partial \Omega) $ to the original Galerkin discretisation \eqref{eq:Galerkina} using $\lambda_{\ell_{\max}}$. This step does not require an additional linear system to be solved.
		\end{itemize}
		
		Naturally, the advantage of this strategy is that the required spectral information of the operator $\widetilde{\mathcal{A} }$ is available from our prior analysis in \cite{Hassan}. 
		
		It remains to perform a convergence analysis of the Krylov subspace solver used to solve the linear system arising from the Galerkin discretisation \eqref{eq:Galerkin}. We now encounter the second difficulty: Since the boundary integral operator $\widetilde{\mathcal{A}}$ is obviously \emph{non-symmetric}, we solve the linear system associated with the Galerkin discretisation \eqref{eq:Galerkin} using the GMRES solver introduced by Saad and Schultz \cite{saad1986}. And the convergence behaviour of GMRES is, in general, considerably more complex than that of the well-known conjugate gradient (CG) method (see, e.g., \cite{TUM} for a comprehensive discussion). 
		
		\vspace{2mm}
		Broadly speaking, the GMRES solver can be applied to four qualitatively different solution matrices $\boldsymbol{A}$:
		
		\vspace{2mm}
		
		\begin{itemize}
			\item[M1)] Symmetric matrices, i.e., $\boldsymbol{A}=\boldsymbol{A}^{\rm T}$: In this case the convergence behaviour of GMRES depends only the spectrum of the solution matrix. Moreover, since the spectrum of the solution matrix is purely real, estimates of the residual at each iteration can be obtained using only the largest and smallest eigenvalues (see, e.g., \cite[Chapter 3]{Fischer}). 		\vspace{2mm}
			
			\item[M2)] Non-symmetric normal matrices, i.e., $\boldsymbol{A}\boldsymbol{A}^{\rm T}=\boldsymbol{A}^{\rm T}\boldsymbol{A}$: In principle, the convergence behaviour of GMRES in this case also depends only the spectrum of the solution matrix. However, since the spectrum of the solution matrix need not be real, useful bounds on the residual at each iteration cannot be obtained using only the eigenvalues with largest and smallest real parts. Instead one typically requires information on the distribution of the full spectrum in the complex plane (see, e.g., \cite[Section 4]{saad1981krylov}). 		\vspace{2mm}
			
			\item[M3)] Non-normal diagonalisable matrices, i.e., $\boldsymbol{A}= \boldsymbol{X}^{-1}\boldsymbol{D}\boldsymbol{X}$ for some diagonal matrix $\boldsymbol{D}$ and non-unitary matrix $\boldsymbol{X}$: In this case, the convergence behaviour of GMRES is typically less known since estimates on the residual at each iteration depend not only on the spectrum of solution matrix but also on the conditioning of the matrix appearing in its diagonalisation, which in general is unknown (see, e.g., \cite[Chapter 6]{Saad}).  		\vspace{2mm}
			
			\item[M4)] Non-diagonalisable matrices: In this case, the convergence behaviour of GMRES is significantly more difficult to analyse and there are only partial theoretical results (see, e.g., \cite{nachtigal1992fast} for an approach based on the pseudospectrum and \cite{eiermann2001geometric, eisenstat1983variational, elman1982iterative} for approaches based on the so-called \emph{field of values}).
		\end{itemize}
		
		\vspace{2mm}
		Consequently, in order to have a reasonable hope of analysing the convergence behaviour of GMRES applied to the linear system arising from the Galerkin discretisation \eqref{eq:Galerkin}, we must show that the matrix discretisation of $\widetilde{\mathcal{A} }$ is either \underline{normal} or \underline{diagonalisable}. Unfortunately, the matrix discretisation of $\widetilde{\mathcal{A} }$ is not normal since the operators $\mathcal{V}$ and $\text{DtN}$ do not commute in general. Consequently, we must prove that the matrix discretisation of $\widetilde{\mathcal{A} }$ is diagonalisable. Notice that since the convergence behaviour of GMRES depends on the conditioning of the matrix appearing in the diagonalisation, the exact choice of diagonalisation is vital. 
		
		\vspace{2mm}

		
		\subsection{The Diagonalisation of the operator $\widetilde{\mathcal{A}}$}\label{sec:3.1} ~
		
		\vspace{2mm}
		
		We will assume throughout this subsection that the discretisation parameter $\ell_{\max} \in \mathbb{N}$ is fixed. Additionally, in order to obtain a useful diagonalisation, we will assume in the sequel that the dielectric function $\kappa$ satisfies one of two conditions:
		
		\vspace{2mm}				
		\begin{enumerate}
			\item Either $\kappa_j > \kappa_0$ for all $j=1, \ldots, N$; \vspace{2mm}
			\item Or $\kappa_j < \kappa_0$ for all $j=1, \ldots, N$.
		\end{enumerate}
		
		\vspace{2mm}
		
		The necessity of the above assumption for the subsequent analysis will be discussed in Remark \ref{rem:New_1}. Let us point out however that from a practical point of view, this additional constraint is not too restrictive since it covers the two cases which typically arise in physical applications: The case of weakly polarisable particles embedded in a highly polarisable solvent and the case of highly polarisable particles embedded in a weakly polarisable medium. Examples of the former include teflon, PMMA, polyethylene or polypropylene particles in water (see \cite{lindgren2018dynamic}). Examples of the later include a wide range of Titanium-based oxides or certain highly polarisable polymer particles in air (see \cite{polymer}).

		\begin{definition}[Finite Dimensional Operators]\label{def:finite}~
			Let $\widetilde{\mathcal{A}} \colon \breve{H}^{\frac{1}{2}}(\partial \Omega) \rightarrow \breve{H}^{\frac{1}{2}}(\partial \Omega)$ be the `modified' boundary integral operator defined through Definition \ref{def:tildeA}, let $\mathcal{V} \colon H^{-\frac{1}{2}}(\partial \Omega) \rightarrow H^{\frac{1}{2}}(\partial \Omega)$ be the single layer boundary integral operator and let $\text{DtN} \colon H^{\frac{1}{2}}(\partial \Omega) \rightarrow H^{-\frac{1}{2}}(\partial \Omega)$ be the Dirichlet-to-Neumann map. Then we  define  the~finite  dimensional operators $\widetilde{\mathcal{A}}_{\ell_{\max}} \colon W_{0}^{\ell_{\max}}(\partial \Omega) \rightarrow W_{0}^{\ell_{\max}}(\partial \Omega) $, ${\mathcal{V}}_{\ell_{\max}} \colon W_{0}^{\ell_{\max}}(\partial \Omega) \rightarrow W_{0}^{\ell_{\max}}(\partial \Omega) $ and $\text{DtN}_{\ell_{\max}} \colon W_{0}^{\ell_{\max}}(\partial \Omega) \rightarrow W_{0}^{\ell_{\max}}(\partial \Omega) $ as
			\begin{align*}
				\widetilde{\mathcal{A} }_{\ell_{\max}}:= \mathbb{P}_{\ell_{\max}} \mathbb{P}_0^{\perp} \widetilde{\mathcal{A}}, \qquad {\mathcal{V} }_{\ell_{\max}}:=\mathbb{P}_{\ell_{\max}} \mathbb{P}_0^{\perp} {\mathcal{V}}, \qquad \text{and} \qquad \text{DtN}_{\ell_{\max}}:=\mathbb{P}_{\ell_{\max}}\text{DtN}. 
			\end{align*}
		\end{definition}
		
		%
		
		\begin{definition}\label{def:DtnK}
			We define the $L^2$-symmetric, positive definite, finite-dimensional operator $\text{DtN}_{\ell_{\max}} ^{\kappa} \colon W^{\ell_{\max}}_0(\partial \Omega)$ $\rightarrow~W^{\ell_{\max}}_0(\partial \Omega) $ as the mapping with the property that for all $\psi_{\ell_{\max}} \in W_0^{\ell_{\max}}(\partial \Omega)$ it holds that
			\begin{align*}
				\text{DtN}^{\kappa}_{\ell_{\max}} \psi_{\ell_{\max}}:=
				\left\vert\frac{\kappa- \kappa_0}{\kappa_0}\right\vert \text{DtN}_{\ell_{\max}}\psi_{\ell_{\max}}.
			\end{align*}
		\end{definition}
		
		Consider Definitions \ref{def:finite} and \ref{def:DtnK} and let $I_{\ell_{\max}}\colon W^{\ell_{\max}}_0(\partial \Omega) \rightarrow W^{\ell_{\max}}_0(\partial \Omega) $ denote the identity mapping. It follows that
		\begin{align}\label{eq:newA}
			\widetilde{\mathcal{A} }_{\ell_{\max}} = \begin{cases}
				I_{\ell_{\max}}  + \mathcal{V}_{\ell_{\max}} \text{DtN}^{\kappa}_{\ell_{\max}} \qquad \text{if } \kappa > \kappa_0,\\
				I_{\ell_{\max}}  - \mathcal{V}_{\ell_{\max}} \text{DtN}^{\kappa}_{\ell_{\max}} \qquad \text{if } \kappa < \kappa_0.
			\end{cases}
		\end{align}
		
		Equation \eqref{eq:newA} suggests a natural diagonalisation strategy: Since $\text{DtN}^{\kappa}_{\ell_{\max}}$ is symmetric positive definite, the associated square root operator $(\text{DtN}_{\ell_{\max}}^{\kappa})^{\frac{1}{2}} \colon W^{\ell_{\max}}_0(\partial \Omega) \rightarrow W^{\ell_{\max}}_0(\partial \Omega) $ can be defined and used to diagonalise~$\widetilde{\mathcal{A} }$.

		\begin{definition}\label{def:Asym}
			We define the $L^2$-symmetric, finite-dimensional operator $\widetilde{\mathcal{A}}^{\rm sym}_{\ell_{\max}} \colon W^{{\ell_{\max}}}_0(\partial \Omega) \rightarrow W_0^{{\ell_{\max}}}(\partial \Omega) $ as 
			\begin{align*}
				\widetilde{\mathcal{A}}^{\rm sym}_{\ell_{\max}}:= \begin{cases}I_{\ell_{\max}}+ (\text{DtN}_{\ell_{\max}}^{\kappa})^{\frac{1}{2}}\mathcal{V}_{\ell_{\max}}(\text{DtN}^{\kappa}_{\ell_{\max}})^{\frac{1}{2}}\quad &\text{ if } \kappa > \kappa_0,\\[1em]
					I_{\ell_{\max}}- (\text{DtN}_{\ell_{\max}}^{\kappa})^{\frac{1}{2}}\mathcal{V}_{\ell_{\max}}(\text{DtN}_{\ell_{\max}}^{\kappa})^{\frac{1}{2}} \quad &\text{ if } \kappa < \kappa_0.
				\end{cases}
			\end{align*}
		\end{definition}
		
		\begin{remark}\label{rem:similar}
			Consider Definition \ref{def:Asym}. We observe immediately that the finite-dimensional operators $\widetilde{\mathcal{A} }_{\ell_{\max}}$ and $\widetilde{\mathcal{A} }_{\ell_{\max}}^{\rm sym}$ are \emph{similar}, i.e., it holds that
			\begin{align}\label{eq:Hassan1}
				\widetilde{\mathcal{A} }_{\ell_{\max}} = (\text{DtN}^{\kappa}_{\ell_{\max}})^{-\frac{1}{2}} \widetilde{\mathcal{A} }_{\ell_{\max}}^{\rm sym}(\text{DtN}^{\kappa}_{\ell_{\max}})^{\frac{1}{2}}.
			\end{align}
			
			Although we have not obtained an explicit diagonalisation of $\widetilde{\mathcal{A} }_{\ell_{\max}}$, we have shown that it is similar to the $L^2$-symmetric operator $\widetilde{\mathcal{A} }_{\ell_{\max}}^{\rm sym}$. In fact our subsequent convergence analysis does not require an explicit diagonalisation and the relation \eqref{eq:Hassan1} will be sufficient.
		\end{remark}
		
		\begin{remark}\label{rem:New_1}
			Consider Definition \ref{def:Asym} and Remark \ref{rem:similar}. The key difficulty in extending such a symmetrisation strategy to the case of a general dielectric function is that if $\kappa$ is such that $\kappa \vert_{\partial \Omega_i} > \kappa_0$ and $\kappa \vert_{\partial \Omega_j} < \kappa_0$ for some $i, j \in \{1, \ldots, N\}$ then the operator $\widetilde{\mathcal{A}}_{\ell_{\max}}$ cannot be written in the straightforward form \eqref{eq:newA} using the operator~$\text{DtN}^{\kappa}_{\ell_{\max}}$. Of course, in this case we may choose to define (c.f., Definition \ref{def:DtnK}), the \emph{modified} finite-dimensional operator  
			\begin{align*}
				\widetilde{\text{DtN}}^{\kappa}_{\ell_{\max}} \psi_{\ell_{\max}}:=\frac{\kappa- \kappa_0}{\kappa_0} \text{DtN}_{\ell_{\max}}\psi_{\ell_{\max}},
			\end{align*}
			in which case, we can indeed write $\widetilde{\mathcal{A} }_{\ell_{\max}} = 
			I_{\ell_{\max}}  + \mathcal{V}_{\ell_{\max}} \widetilde{\text{DtN}}^{\kappa}_{\ell_{\max}}$, but the modified operator $\widetilde{\text{DtN}}^{\kappa}_{\ell_{\max}}$ used in this construction is \emph{indefinite} and therefore its square root cannot be defined.
		\end{remark}
		
		Since the finite-dimensional operator $\widetilde{\mathcal{A} }_{\ell_{\max}}^{\rm sym}$ will feature prominently in our subsequent convergence analysis, we must obtain bounds on the spectrum of this operator. Notice that $\widetilde{\mathcal{A} }^{\rm sym}_{\ell_{\max}}$ is $L^2$-symmetric, and thus has purely real eigenvalues. The next lemma gives bounds on the smallest and largest eigenvalues of this operator.
		
		\begin{lemma}\label{lem:eigenvalues}
			Let the symmetric, finite-dimensional operator $\widetilde{\mathcal{A} }_{\ell_{\max}}^{\rm sym} \colon W_0^{\ell_{\max}}(\partial \Omega) \rightarrow W_0^{\ell_{\max}}(\partial \Omega)$ be defined as in Definition \ref{def:Asym}, let the continuity constant $C_{\widetilde{\mathcal{A}}}$ of the operator $\widetilde{\mathcal{A}}$ be defined as in Theorem \ref{lem:well-posed}, let the constant $\alpha_0 \in \mathbb{R}$  be defined as
			\begin{align*}
				\alpha_0:= \begin{cases}
					1 \qquad &\text{ if } \kappa > \kappa_0,\\
					\min \frac{\kappa}{\kappa_0} \qquad &\text{ if } \kappa < \kappa_0,
				\end{cases} 
			\end{align*}
			and let $\mu_0 \in \mathbb{R}$ and $\mu_{\max} \in \mathbb{R}$ denote the smallest and largest eigenvalue respectively of $\widetilde{\mathcal{A} }_{\ell_{\max}}^{\rm sym}$. Then it holds that
			\begin{align*}
				\mu_0 \geq \alpha_0 \qquad \text{and} \qquad \mu_{\max} \leq C_{\widetilde{\mathcal{A}}}.
			\end{align*}
		\end{lemma}
		\begin{proof}
			We first prove the bound for the largest eigenvalue $\mu_{\max} \in \mathbb{R}$. Let $\psi^{\ell_{\max}}_{\max} \in W_0^{\ell_{\max}}(\partial \Omega)$ denote the eigenfunction corresponding to $\mu_{\max}$. We then have by the similarity of the operators $\widetilde{\mathcal{A}}_{\ell_{\max}}$ and $\widetilde{\mathcal{A} }^{\rm sym}_{\ell_{\max}}$ (see Remark \ref{rem:similar}) that
			\begin{align*}
				\mu_{\max} \psi^{\ell_{\max}}_{\max}= \widetilde{\mathcal{A} }^{\rm sym}_{\ell_{\max}} \psi^{\ell_{\max}}_{\max} = (\text{DtN}^{\kappa}_{\ell_{\max}})^{\frac{1}{2}} \widetilde{\mathcal{A} }_{\ell_{\max}}(\text{DtN}^{\kappa}_{\ell_{\max}})^{-\frac{1}{2}}\psi^{\ell_{\max}}_{\max}.
			\end{align*} 
			
			Define the function $\phi_{\max}^{\ell_{\max}}:= (\text{DtN}^{\kappa})^{-\frac{1}{2}} \psi_{\max}^{\ell_{\max}}$. A straightforward calculation then yields
			\begin{align*}
				\mu_{\max} \phi_{\max}^{\ell_{\max}} = \widetilde{\mathcal{A}}_{\ell_{\max}} \phi_{\max}^{\ell_{\max}}, \quad \implies \quad
				\mu_{\max}^2 \big| \big|\big| \phi_{\max}^{\ell_{\max}} \big| \big|\big|^2 = \big| \big|\big| \widetilde{\mathcal{A}}_{\ell_{\max}} \phi_{\max}^{\ell_{\max}}\big|\big| \big|^2.
			\end{align*}
			
			Consequently, we obtain from Theorem \ref{lem:well-posed} that $\mu_{\max} \leq C_{\widetilde{\mathcal{A}}}$.
			
			Next, we prove the estimate for the smallest eigenvalue $\mu_0 \in \mathbb{R}$. Let $\psi^{\ell_{\max}}_{0} \in W_0^{\ell_{\max}}(\partial \Omega)$ denote the eigenfunction corresponding to $\mu_0$ and define ${\phi}_0^{\ell_{\max}}:=(\text{DtN}^{\kappa})^{-\frac{1}{2}}\psi^{\ell_{\max}}_0$. A similar calculation as above yields that
			\begin{align} \label{eq:New1}
				\mu_0 {\phi}_0^{\ell_{\max}} = \widetilde{\mathcal{A}}_{\ell_{\max}}  {\phi}_0^{\ell_{\max}}. 
			\end{align}
			
			Consider first the case $\kappa > \kappa_0$. We obtain from Equation \eqref{eq:New1} and the definition of $\widetilde{\mathcal{A}}_{\ell_{\max}}$ that
			\begin{align*}
				\mu_0 \left({\phi}_0^{\ell_{\max}}, \text{DtN}_{\ell_{\max}}^{\kappa}{\phi}_0^{\ell_{\max}}\right)_{L^2(\partial \Omega)}&=\left({\phi}_0^{\ell_{\max}}, \text{DtN}_{\ell_{\max}}^{\kappa}{\phi}_0^{\ell_{\max}}\right)_{L^2(\partial \Omega)}+\left(\mathcal{V}\text{DtN}_{\ell_{\max}}^{\kappa}{\phi}_0^{\ell_{\max}}, \text{DtN}_{\ell_{\max}}^{\kappa}{\phi}_0^{\ell_{\max}}\right)_{L^2(\partial \Omega)} \\
				&\geq\left({\phi}_0^{\ell_{\max}}, \text{DtN}_{\ell_{\max}}^{\kappa}{\phi}_0^{\ell_{\max}}\right)_{L^2(\partial \Omega)},
			\end{align*}
			where the second step uses the coercivity of $\mathcal{V}$ (see Property 1 in Section \ref{sec:2a}). This yields $\mu_0 \geq 1$
			
			Next, consider the case $\kappa < \kappa_0$. We obtain from Equation \eqref{eq:New1} and the definition of $\widetilde{\mathcal{A}}_{\ell_{\max}}$ that
			\begin{align}\label{eq:New2}
				(\mu_0-1) \left({\phi}_0^{\ell_{\max}}, \text{DtN}_{\ell_{\max}}^{\kappa}{\phi}_0^{\ell_{\max}}\right)_{L^2(\partial \Omega)}&=-\left(\mathcal{V}\text{DtN}_{\ell_{\max}}^{\kappa}{\phi}_0^{\ell_{\max}}, \text{DtN}_{\ell_{\max}}^{\kappa}{\phi}_0^{\ell_{\max}}\right)_{L^2(\partial \Omega)}.
			\end{align}
			
			We define the function ${\phi}_0^{\kappa, \ell_{\max}}:= \left\vert \frac{\kappa-\kappa_0}{\kappa_0} \right\vert \phi_0^{\ell_{\max}}$. Using the Calderon identity $\text{DtN}= W + \text{DtN}\mathcal{V}\text{DtN}$ (see, e.g., \cite[Theorem 3.8.7]{Schwab}) we see that
			\begin{align}\nonumber
				\left(\mathcal{V}\text{DtN}_{\ell_{\max}}^{\kappa}{\phi}_0^{\ell_{\max}}, \text{DtN}_{\ell_{\max}}^{\kappa}{\phi}_0^{\ell_{\max}}\right)_{L^2(\partial \Omega)} &=\left(\text{DtN}_{\ell_{\max}}{\phi}_0^{\kappa, \ell_{\max}},{\phi}_0^{\kappa, \ell_{\max}}\right)_{L^2(\partial \Omega)}-\left(\mathcal{W}{\phi}_0^{\kappa, \ell_{\max}}, {\phi}_0^{\kappa, \ell_{\max}}\right)_{L^2(\partial \Omega)}\\
				&\leq \left(\text{DtN}_{\ell_{\max}}{\phi}_0^{\kappa, \ell_{\max}}, {\phi}_0^{\kappa, \ell_{\max}}\right)_{L^2(\partial \Omega)}, \label{eq:Review_1}
			\end{align}
			where the inequality follows from the non-negativity of the hypersingular operator $\mathcal{W}$ (see, e.g., Property 2 in Section \ref{sec:2a}). Using this bound in Equation \eqref{eq:New2} then yields
			\begin{align*}
				(\mu_0-1) \sum_{j=1}^N  \left\vert\frac{\kappa_j-\kappa_0}{\kappa_0}\right\vert \left(\text{DtN}_{\ell_{\max}}{\phi}_0^{\ell_{\max}}, {\phi}_0^{\ell_{\max}}\right)_{L^2(\partial \Omega_j)} \geq -\sum_{j=1}^N  \left\vert\frac{\kappa_j-\kappa_0}{\kappa_0}\right\vert^2  \left(\text{DtN}_{\ell_{\max}}{\phi}_0^{\ell_{\max}}, {\phi}_0^{\ell_{\max}}\right)_{L^2(\partial \Omega_j)}.
			\end{align*}
			
			Simple calculus and the fact that $\kappa < \kappa_0$ by assumption, allows us to conclude that $\mu_0 \geq \min_{j=1, \ldots, N} \frac{\kappa_j}{\kappa_0}$.
		\end{proof}

		We now have all the ingredients necessary to analyse the convergence behaviour of GMRES applied to the linear system arising from the ``reduced'' Galerkin discretisation \eqref{eq:Galerkin}. \vspace{2mm}

		\subsection{GMRES Convergence Analysis and Solution Strategy.}\label{sec:3.2}~
	
		\vspace{2mm}
		
		We begin this subsection by fixing some additional notation. As a first step we would like to write explicitly the linear system arising from the Galerkin discretisation \eqref{eq:Galerkin}. In view of Definition \ref{def:Appromxation2} of our approximation space~$W_0^{\ell_{\max}}(\partial \Omega)$, the natural choice of basis functions are the local spherical harmonics on each sphere. 
		
		\begin{definition}[Choice of Basis]\label{def:Basis}~
			Let $\ell_{\max} \in \mathbb{N}$. We denote for each $j \in \{1, \ldots, N\}$ and all $\ell \in \{1, \ldots, \ell_{\max}\}$, $-\ell \leq m \leq \ell$ the function ${\mathcal{Y}}^j_{\ell m} \colon \partial \Omega \rightarrow \mathbb{R}$ defined as
			\begin{align*}
				{\mathcal{Y}}^j_{\ell m}(\bold{x}):= \begin{cases}\mathcal{Y}_{\ell m} \left(\frac{\bold{x}-\bold{x}_j}{\vert \bold{x}-\bold{x}_j\vert}\right) \quad &\text{for all } \bold{x} \in \partial \Omega_j,\\
					0 \quad &\text{otherwise},
				\end{cases}
			\end{align*}
			and we equip the approximation space $W_0^{\ell_{\max}}(\partial \Omega)$ with the basis $\{\mathcal{Y}^j_{\ell m}\}$.
		\end{definition}
		
		\noindent	\textbf{Notation:} Let $\ell_{\max} \in \mathbb{N}$. We will henceforth denote by $M:= N \cdot (\ell_{\max}+1)^2-N$, the dimension of the approximation space $W_0^{\ell_{\max}}(\partial \Omega)$. Notice that the dimension of the space $W^{\ell_{\max}}(\partial \Omega)$ is then given by $M+N$.

		\begin{remark}
			Consider Definition \ref{def:Basis} of the basis functions on $W_0^{\ell_{\max}}(\partial \Omega)$. These functions establish an isomorphism between $W_0^{\ell_{\max}}(\partial \Omega)$ and $\mathbb{R}^{M}$. Indeed, we associate an arbitrary $\psi \in W_0^{\ell_{\max}}(\partial \Omega)$ with $\boldsymbol{\psi} \in \mathbb{R}^{M}$ defined as
			\begin{align*}
				[\boldsymbol{\psi}_{i}]_{\ell}^m:= \left(\psi, \mathcal{Y}^j_{\ell m}\right)_{L^2(\partial \Omega_i)}, \qquad \text{for } i\in \{1, \ldots, N\}, \quad \ell \in \{1, \ldots, \ell_{\max}\} \quad \text{and } \hspace{1mm} |m| \leq \ell.
			\end{align*}
			
			Consequently, given functions in the space $W_0^{\ell_{\max}}(\partial \Omega)$, we will often refer to their vector representations in $\mathbb{R}^{M}$ and vice versa. Moreover, to facilitate identification we will frequently use bold symbols for the vector representations.
		\end{remark}

		\begin{definition}\label{def:Matrices}
			Let $\ell_{\max} \in \mathbb{N}$, let $\mathbb{Q}_{\ell_{\max}} \colon H^{-\frac{1}{2}}(\partial \Omega) \rightarrow W^{\ell_{\max}}(\partial \Omega) $ denote the orthogonal projection operator, and let $\sigma_f \in H^{-\frac{1}{2}}(\partial \Omega)$. Then 
			\begin{itemize}
				\item We define the right-hand side vector $\boldsymbol{\sigma_f} \in \mathbb{R}^{M}$ as
				\begin{align*}
					[\boldsymbol{\sigma_f}_{i}]_{\ell}^m := \frac{4\pi}{\kappa_0}\left(\mathcal{V}\mathbb{Q}_{\ell_{\max}}\sigma_f, \mathcal{Y}^i_{\ell m}\right)_{L^2(\partial \Omega_i)}, \qquad \text{for } i\in \{1, \ldots, N\}, \quad \ell \in \{1, \ldots, \ell_{\max}\} \quad \text{and } \hspace{1mm} |m| \leq \ell,
				\end{align*}
				
				\item We define the diagonal positive definite matrix $\boldsymbol{\rm DtN}^{\kappa} \in \mathbb{R}^{M\times M}$ as 
				\begin{align*}
					[\boldsymbol{\rm DtN}^{\kappa}_{ij}]_{\ell \ell'}^{m m'}:= \left( (\text{DtN}^{\kappa}_{\ell_{\max}})^{\frac{1}{2}}\mathcal{Y}^j_{\ell'm'}, \mathcal{Y}^i_{\ell m}\right)_{L^2(\partial \Omega_i)}, \hspace{3mm}\text{for } i, j\in \{1, \ldots, N\}, ~~ \ell, \ell' \in \{1, \ldots, \ell_{\max}\} \quad \text{and } \hspace{1mm} \vert m \vert, \vert m' \vert  \leq \ell,
				\end{align*}
				
				\item We define the solution matrix $\boldsymbol{A} \in \mathbb{R}^{M\times M}$ as
				\begin{align*}
					[\boldsymbol{A}_{ij}]_{\ell \ell'}^{m m'}:= \left(\widetilde{\mathcal{A}}_{\ell_{\max}}\mathcal{Y}^j_{\ell'm'}, \mathcal{Y}^i_{\ell m}\right)_{L^2(\partial \Omega_i)}, \quad \text{for } i, j\in \{1, \ldots, N\}, \quad \ell, \ell' \in \{1, \ldots, \ell_{\max}\} \quad \text{and } \hspace{1mm} \vert m \vert, \vert m' \vert  \leq \ell,
				\end{align*}
				
				\item We define the symmetrised solution matrix $\boldsymbol{A}^{\rm sym} \in \mathbb{R}^{M\times M}$ as
				\begin{align*}
					[\boldsymbol{A}^{\rm sym}_{ij}]_{\ell \ell'}^{m m'}:= \left(\widetilde{\mathcal{A}}^{\rm sym}_{\ell_{\max}}\mathcal{Y}^j_{\ell'm'}, \mathcal{Y}^i_{\ell m}\right)_{L^2(\partial \Omega_i)}, \hspace{5mm} \text{for } i, j\in \{1, \ldots, N\}, \quad \ell, \ell' \in \{1, \ldots, \ell_{\max}\} \quad\text{and } \hspace{1mm} \vert m \vert, \vert m' \vert  \leq \ell,
				\end{align*}
			\end{itemize}
		\end{definition}
		
		Two remarks are now in order.
		
		\begin{remark}\label{rem:Sym}
			Consider Definition \ref{def:Matrices}. A direct calculation shows that the matrices $\boldsymbol{A}$ and $\boldsymbol{A}^{\rm sym}$ are similar and we have
			\begin{align*}
				\boldsymbol{A}= (\boldsymbol{\rm DtN}^{\kappa})^{-1} \boldsymbol{A}^{\rm sym}\boldsymbol{\rm DtN}^{\kappa}.
			\end{align*}
		\end{remark}
		
		\begin{remark}
			Consider Definition \ref{def:Matrices}. We emphasise that the solution matrix $\boldsymbol{A}$ and the symmetrised solution matrix $\boldsymbol{A}^{\rm sym}$ are nothing else than the representation in the basis of local spherical harmonics functions of the finite-dimensional operators $\widetilde{\mathcal{A} }_{\ell_{\max}}$ and $\widetilde{\mathcal{A} }^{\rm sym}_{\ell_{\max}}$ defined through Definitions \ref{def:tildeA} and \ref{def:Asym} respectively. We can therefore write the ``reduced'' Galerkin discretisation \eqref{eq:Galerkin} as expected in matrix form.
		\end{remark}

		\noindent{\textbf{Matrix Formulation of the ``Reduced'' Galerkin Discretisation \eqref{eq:Galerkin}}}~
		
		Let $\ell_{\max} \in \mathbb{N}$, let $\sigma_f \in H^{-\frac{1}{2}}(\partial \Omega)$, and let the vector $\boldsymbol{\sigma_f} \in \mathbb{R}^{M}$ and the matrix $\boldsymbol{A} \in \mathbb{R}^{M \times M}$ be defined as in Definition \ref{def:Matrices}. Find a vector $\boldsymbol{\lambda} \in \mathbb{R}^{M}$ such that
		\begin{align}\label{eq:Matrix}
			\boldsymbol{A}\boldsymbol{\lambda}= \boldsymbol{\sigma_f}.
		\end{align}
		
		We are now ready to state our main convergence result.
		
		\begin{theorem}[Convergence Analysis for GMRES-based Strategy]\label{thm:3}~\\
			Let $\ell_{\max} \in \mathbb{N}$, let $\sigma_f \in H^{-\frac{1}{2}}(\partial \Omega)$, let $\nu_{\ell_{\max}} \in W^{\ell_{\max}}(\partial \Omega) $ be the unique solution to the Galerkin discretisation defined through Equation \eqref{eq:Galerkina} with right hand side given by $\sigma_f$, and let $\mathbb{P}_0^{\perp} \colon H^{\frac{1}{2}}(\partial \Omega) \rightarrow \breve{H}^{\frac{1}{2}}(\partial \Omega)$ and $\mathbb{Q}_{\ell_{\max}}\colon H^{\frac{1}{2}}(\partial \Omega) \rightarrow  W^{\ell_{\max}}(\partial \Omega) $ be the orthogonal projection operators. Then for every $\epsilon > 0$ there exists a function $\nu_{\ell_{\max}}^{\rm{approx}} \in W^{\ell_{\max}}(\partial \Omega) $ and a natural number $R_{\epsilon}>0$ that does not depend on the number of dielectric spheres $N$ such that $\nu_{\ell_{\max}}^{\rm{approx}} $ can be computed using at most $R_{\epsilon}$ iterations of GMRES and such that the following error estimate holds
			\begin{align*}
				\frac{||| \nu_{\ell_{\max}}^{\rm{approx}}  - \nu_{\ell_{\max}}|||^*}{||| \mathbb{P}_0^{\perp}\nu_{\ell_{\max}}|||^*+ \frac{4\pi}{\kappa_0}||| \mathbb{P}_0^{\perp}\mathbb{Q}_{\ell_{\max}}\sigma_f|||^*}\,  < \epsilon.
			\end{align*}
		\end{theorem}
		\begin{proof}
			
			We begin by defining the affine transformation $\mathcal{B} \colon W_0^{\ell_{\max}}(\partial \Omega) \rightarrow W^{\ell_{\max}}(\partial \Omega) $ as the mapping with the property that for all $\psi_{\ell_{\max}} \in W_0^{\ell_{\max}}(\partial \Omega)$ it holds that
			\begin{align}\label{eq:thm31}
				\mathcal{B} \psi_{\ell_{\max}}= \frac{\kappa_0-\kappa}{\kappa_0} \text{DtN} \psi_{\ell_{\max}} + \frac{4\pi}{\kappa_0}\mathbb{Q}_{\ell_{\max}}\sigma_f.
			\end{align}
			
			Let $\lambda_{\ell_{\max}} \in W^{\ell_{\max}}_0(\partial \Omega) $ denote the solution to the ``reduced' Galerkin discretisation \eqref{eq:Galerkin}. It follows from Lemma \ref{lem:indirect} that $\mathcal{B} \lambda_{\ell_{\max}}  = \nu_{\ell_{\max}}$. Thus for any function $\psi_{\ell_{\max}}\in W^{\ell_{\max}}_0(\partial \Omega) $ it holds that
			\begin{align*}
				||| \nu_{\ell_{\max}} - \mathcal{B} \psi_{\ell_{\max}} |||^* &= \Big|\Big|\Big|  \frac{\kappa_0-\kappa}{\kappa_0} \text{DtN} \Big(\lambda_{\ell_{\max}}- \psi_{\ell_{\max}}\Big) \Big|\Big|\Big|^*\\
				&\leq \max \Big \vert \frac{\kappa_0-\kappa}{\kappa_0}\Big \vert ||| \lambda_{\ell_{\max}}- \psi_{\ell_{\max}} |||.
			\end{align*}
			
			Let $\epsilon> 0$ be fixed and let $\widetilde{\epsilon}:=  \frac{\epsilon}{\max \left \vert \frac{\kappa_0-\kappa}{\kappa_0}\right \vert}$. It follows that 
			\begin{align}\label{eq:thm32}
				\vert ||| \lambda_{\ell_{\max}}- \psi_{\ell_{\max}} ||| < \widetilde{\epsilon} \implies ||| \nu_{\ell_{\max}} - \mathcal{B} \psi_{\ell_{\max}} |||^* < \epsilon.
			\end{align}

			We will therefore show that there exists a natural number $R_{{\epsilon}}$ that is independent of the number $N$ of open balls such that one can compute a function $\psi_{\ell_{\max}} \in W_0^{\ell_{\max}}(\partial \Omega)$ using at most $R_{\epsilon}$ iterations of GMRES and such that
			\begin{align}\label{eq:thm33}
				\vert ||| \lambda_{\ell_{\max}}- \psi_{\ell_{\max}} ||| < \widetilde{\epsilon}.
			\end{align}
			
			This will allow us to define $\nu_{\ell_{\max}}^{\rm approx}:=\mathcal{B}\psi_{\ell_{\max}}$ and hence complete the proof.
			
			Consider the matrix Equation \eqref{eq:Matrix}. Let $\boldsymbol{\lambda}_0 \in \mathbb{R}^M$ be some initialisation and let $\boldsymbol{\lambda}_k \in \mathbb{R}^{M}, ~ k \in \mathbb{N}$ denote the $k^{\rm{th}}$ iterate generated by GMRES applied to this linear system. Next, let $\boldsymbol{r}_k:= \boldsymbol{\sigma_f}- \boldsymbol{A}  \boldsymbol{\lambda}_k, ~ k \in \mathbb{N}_0$ denote the $k^{\text{th}}$ residual. It is well known (see, e.g., \cite{Greenbaum} or \cite{TUM}) that for all natural numbers $k \in \mathbb{N}$ the GMRES residual $\boldsymbol{r}_k$ satisfies
			\begin{align}\label{eq:thm34}
				\Vert \boldsymbol{r}_k\Vert_{\ell^2}= \min_{p \in \pi_k} \big\Vert p\left(\boldsymbol{A} \right)\boldsymbol{r}_0\big\Vert_{\ell^2} \leq \min_{p \in \pi_k} \big\Vert p\left(\boldsymbol{A}\right)\big\Vert_{2}\big\Vert \boldsymbol{r}_0\big\Vert_{\ell^2},
			\end{align}
			where $\pi_k$ denotes the set of all polynomials $p \colon \mathbb{R}\rightarrow \mathbb{R}$ of degree at most $k$ such that $p(0)=1$, $\Vert \cdot \Vert_{\ell^2}$ denotes the standard Euclidean norm in $\mathbb{R}^M$ and $\Vert \cdot \Vert_{2}$ denotes the matrix norm induced by the standard Euclidean norm in $\mathbb{R}^M$. 
			
			Using Remark \ref{rem:Sym} and the fact that the matrix $\boldsymbol{A}^{\rm sym}$ is symmetric, the bound \eqref{eq:thm34} can be simplified (see, e.g., \cite{TUM}) to obtain
			\begin{align*}
				\Vert \boldsymbol{r}_k\Vert_{\ell^2}&\leq \min_{p \in \pi_k} \big\Vert p\left(\boldsymbol{A}^{\rm sym}\right)\big\Vert_{2}\big \Vert (\boldsymbol{\rm DtN}^{\kappa})^{-1}\big \Vert_{2} \big \Vert \boldsymbol{\rm DtN}^{\kappa}\big \Vert_{2}\big\Vert \boldsymbol{r}_0\big\Vert_{\ell^2},\\
				&\leq \min_{p \in \pi_k} \max_n\big\vert p\left(\mu_n\right)\big\vert\big \Vert (\boldsymbol{\rm DtN}^{\kappa})^{-1}\big \Vert_{2} \big \Vert \boldsymbol{\rm DtN}^{\kappa}\big \Vert_{2}\big\Vert \boldsymbol{r}_0\big\Vert_{\ell^2},
			\end{align*}		
			where $\mu_n, n =0, \ldots, M-1$ denote the (ascendingly ordered) eigenvalues of the matrix $\boldsymbol{A}^{\rm sym}$. We now simplify each term in this estimate.
			
			Due to Lemma \ref{lem:eigenvalues} we know that all eigenvalues of $\boldsymbol{A}^{\rm sym}$ are positive and we have explicit bounds on the smallest eigenvalue $\mu_0$ and the largest eigenvalue $\mu_{M-1}$ of this matrix. Consequently, we can employ the standard approach of using Chebyshev polynomials of the first kind to estimate the min-max problem (see, e.g., \cite[Chapter 3]{Fischer}). We thus obtain
			\begin{align*}
				\min_{p \in \pi_k} \max_n\big\vert p\left(\mu_n\right)\big\vert \leq 2\left(\frac{\sqrt{\frac{C_{\widetilde{\mathcal{A}}}}{\alpha_0} } -1}{\sqrt{\frac{C_{\widetilde{\mathcal{A}}}}{\alpha_0} } +1}\right)^k,
			\end{align*}
			where $\alpha_0, C_{\widetilde{\mathcal{A}}}$ are bounds on the smallest and largest eigenvalues of $\boldsymbol{A}^{\rm sym}$ respectively, as computed in Lemma~\ref{lem:eigenvalues}.
			
			Next, note that by definition, each non-zero entry of the diagonal positive definite matrix $\boldsymbol{\rm DtN}^{\kappa}$ is given by
			\begin{align*}
				[\boldsymbol{\rm DtN}^{\kappa}_{jj}]_{\ell \ell}^{m m} &= \left\vert \frac{\kappa_j-\kappa_0}{\kappa_0}\right\vert^{\frac{1}{2}}\left((\text{DtN}_{\ell_{\max}})^{\frac{1}{2}}\mathcal{Y}^j_{\ell m}, \mathcal{Y}^j_{\ell m}\right)_{L^2(\partial \Omega_j)}\\
				&= \left\vert \frac{\kappa_j-\kappa_0}{\kappa_0}\right\vert^{\frac{1}{2}} r_j^2\sqrt{\frac{{\ell}}{{r_j}}}, \hspace{2cm} \text{for } j\in \{1, \ldots, N\}, \quad \ell \in \{1, \ldots, \ell_{\max}\} \quad \text{and } -\ell \leq m \leq \ell.
			\end{align*}
			
			If we denote by $\chi_0, \chi_{\max}\in \mathbb{R}$ the smallest and largest entry respectively of the matrix~$\boldsymbol{\rm DtN}^{\kappa}$ we have
			\begin{align*}
				\chi_0 = \min \left \vert \frac{\kappa-\kappa_0}{\kappa_0}\right\vert^{\frac{1}{2}} \min_{j=1\ldots, N} r_j^{\frac{3}{2}}, \hspace{1cm} \text{and} \hspace{1cm} \chi_{\max} = \max \left \vert \frac{\kappa-\kappa_0}{\kappa_0}\right \vert^{\frac{1}{2}}	\max_{j=1\ldots, N} r_j^{\frac{3}{2}}\sqrt{\ell_{\max}}.
			\end{align*}
			
			We thus obtain that
			\begin{align*}
				\big \Vert (\boldsymbol{\rm DtN}^{\kappa})^{-1}\big \Vert_{2}= \frac{1}{\chi_0} =\frac{1}{\min \left \vert \frac{\kappa-\kappa_0}{\kappa_0}\right \vert^{\frac{1}{2}}\min_{j=1\ldots, N} r_j^{\frac{3}{2}}}\hspace{0.5cm} \text{and} \hspace{0.5cm} \big \Vert \boldsymbol{\rm DtN}^{\kappa}\big \Vert_{2}=\chi_{\max}=\max \left \vert \frac{\kappa-\kappa_0}{\kappa_0}\right \vert^{\frac{1}{2}}\max_{j=1\ldots, N} r_j^{\frac{3}{2}}\sqrt{\ell_{\max}}.
			\end{align*}

			The estimate \eqref{eq:thm34} can therefore be bounded as
			\begin{align*}
				\Vert \boldsymbol{r}_k\Vert_{\ell^2} \leq 2\sqrt{\ell_{\max}}\left(\frac{\sqrt{\frac{C_{\widetilde{\mathcal{A}}}}{\alpha_0} } -1}{\sqrt{\frac{C_{\widetilde{\mathcal{A}}}}{\alpha_0} } +1}\right)^k  \left(\frac{\max \left \vert {\kappa-\kappa_0}\right \vert\max_{j=1\ldots, N} r_j^{{3}}}{\min \left \vert {\kappa-\kappa_0}\right \vert\min_{j=1\ldots, N} r_j^{{3}}}\right)^{\frac{1}{2}}\Vert \boldsymbol{r}_0\Vert_{\ell^2}.
			\end{align*}
			
			Next, for each $k \in \mathbb{N}_0$ we denote by $\lambda^{\ell_{\max}}_k \in W_0^{\ell_{\max}}(\partial \Omega)$ the function associated with the vector $\boldsymbol{\lambda}_k \in \mathbb{R}^{M}$. Using the fact that $\boldsymbol{\sigma_f}= \boldsymbol{A}\boldsymbol{\lambda}$ and Definition~\ref{def:Matrices}, we obtain by a direct calculation that
			\begin{align*}
				\Vert \boldsymbol{r}_k\Vert_{\ell^2}  &= \Vert \boldsymbol{A}\boldsymbol{\lambda}-\boldsymbol{A} \boldsymbol{\lambda}_k\Vert_{\ell^2} \geq \min_{j=1, \ldots, N} r_j \left \Vert  \widetilde{\mathcal{A}}_{\ell_{\max}} \lambda_{\ell_{\max}}- \widetilde{\mathcal{A}}_{\ell_{\max}}\lambda_k^{\ell_{\max}}\right\Vert_{L^2(\partial \Omega)}\\
				&\geq \left(\frac{\min_{j=1, \ldots, N} r_j^{{3}}}{\ell_{\max}}\right)^{\frac{1}{2}} \Big|\Big|\Big| \widetilde{\mathcal{A}}_{\ell_{\max}} \lambda_{\ell_{\max}}- \widetilde{\mathcal{A}}_{\ell_{\max}}\lambda_k^{\ell_{\max}}\Big|\Big|\Big| \geq \beta_{\widetilde{\mathcal{A}}}\left(\frac{\min_{j=1, \ldots, N} r_j^{{3}}}{\ell_{\max}}\right)^{\frac{1}{2}} \Big|\Big|\Big| \lambda_{\ell_{\max}}- \lambda_k^{\ell_{\max}}\Big|\Big|\Big|,
			\end{align*}
			where the last step follows from Theorem \ref{lem:well-posed}. In a similar fashion if we pick the initialisation $\boldsymbol{\lambda}_0 \equiv 0$ we have
			\begin{align*}
				\Vert \boldsymbol{r}_0\Vert_{\ell^2} \leq C_{\widetilde{\mathcal{A} }}\max_{j=1\ldots, N} r^{\frac{3}{2}}_j ||| \lambda_{\ell_{\max}}|||.
			\end{align*}

			It therefore follows that
			\begin{align}\label{eq:thm35}
				\Big|\Big|\Big| \lambda_{\ell_{\max}}- \lambda_k^{\ell_{\max}}\Big|\Big|\Big| \leq \frac{2C_{\widetilde{\mathcal{A} }}\ell_{\max} }{\beta_{\tilde{\mathcal{A}}}}\frac{\max_{j=1\ldots, N} r_j^{{3}}}{\min_{j=1, \ldots, N} r_j^3}\left(\frac{\sqrt{\frac{C_{\widetilde{\mathcal{A}}}}{\alpha_0} } -1}{\sqrt{\frac{C_{\widetilde{\mathcal{A}}}}{\alpha_0} } +1}\right)^k  \left(\frac{\max \left \vert {\kappa-\kappa_0}\right \vert }{\min \left \vert {\kappa-\kappa_0}\right \vert}\right)^{\frac{1}{2}} ||| \lambda_{\ell_{\max}}|||.
			\end{align}
			
			Finally, in view of Equation \eqref{eq:thm31} we observe that
			\begin{align*}
				||| \lambda_{\ell_{\max}}||| &= \Big|\Big|\Big| \frac{\kappa_0}{\kappa_0-\kappa} \text{DtN}^{-1}\mathbb{P}_0^{\perp}\Big(\nu_{\ell_{\max}} - \frac{4\pi}{\kappa_0}\mathbb{Q}_{\ell_{\max}}\sigma_f\Big)\Big|\Big|\Big| = \Big|\Big|\Big| \frac{\kappa_0}{\kappa_0-\kappa} \mathbb{P}_0^{\perp}\Big(\nu_{\ell_{\max}} - \frac{4\pi}{\kappa_0}\mathbb{Q}_{\ell_{\max}}\sigma_f\Big)\Big|\Big|\Big|^*\\
				&\leq \max \left \vert \frac{\kappa_0}{\kappa-\kappa_0}\right \vert \left(||| \mathbb{P}_0^{\perp}\nu_{\ell_{\max}}|||^*+ \frac{4\pi}{\kappa_0}||| \mathbb{P}_0^{\perp}\mathbb{Q}_{\ell_{\max}}\sigma_f|||^*\right).
			\end{align*}
			
			Therefore, if we define the constant $\Upsilon_{\rm GMRES}:= \frac{2C_{\widetilde{\mathcal{A} }} }{\beta_{\tilde{\mathcal{A}}}}\frac{\max_{j=1\ldots, N} r_j^{{3}}}{\min_{j=1, \ldots, N} r_j^3} \left(\frac{\max \left \vert {\kappa-\kappa_0}\right \vert }{\min \left \vert {\kappa-\kappa_0}\right \vert}\right)^{\frac{1}{2}} \max \left \vert \frac{\kappa_0}{\kappa-\kappa_0}\right \vert $, then the bound \eqref{eq:thm35} can be written succinctly as
			\begin{align*}
				\Big|\Big|\Big| \lambda_{\ell_{\max}}- \lambda_k^{\ell_{\max}}\Big|\Big|\Big| \leq  \Upsilon_{\rm GMRES}\ell_{\max} \left(\frac{\sqrt{\frac{C_{\widetilde{\mathcal{A}}}}{\alpha_0} } -1}{\sqrt{\frac{C_{\widetilde{\mathcal{A}}}}{\alpha_0} } +1}\right)^k \left(||| \mathbb{P}_0^{\perp}\nu_{\ell_{\max}}|||^*+ \frac{4\pi}{\kappa_0}||| \mathbb{P}_0^{\perp}\mathbb{Q}_{\ell_{\max}}\sigma_f|||^*\right).
			\end{align*}
			
			Consequently, we can define the natural number $R_{\epsilon}$ as
			\begin{align}\label{eq:iterations}
				R_{\epsilon}:=\ceil*{\frac{\log\left(\frac{\widetilde{\epsilon}}{\ell_{\max}\Upsilon_{\rm GMRES}}\right)}{\log\left(\frac{\sqrt{\frac{C_{\widetilde{\mathcal{A}}}}{\alpha_0} } -1}{\sqrt{\frac{C_{\widetilde{\mathcal{A}}}}{\alpha_0} } +1}\right)}}.
			\end{align}
			
			We then obtain
			\begin{align*}
				\Big|\Big|\Big| \lambda_{\ell_{\max}}- \lambda_{R_{\epsilon}}^{\ell_{\max}}\Big|\Big|\Big| < \widetilde{\epsilon} \left(||| \mathbb{P}_0^{\perp}\nu_{\ell_{\max}}|||^*+ \frac{4\pi}{\kappa_0}||| \mathbb{P}_0^{\perp}\mathbb{Q}_{\ell_{\max}}\sigma_f|||^*\right),
				\intertext{which yields using Inequality \eqref{eq:thm32}}
				\frac{||| \nu_{\ell_{\max}} - \mathcal{B} \lambda_{R_{\epsilon}}^{\ell_{\max}}|||^* }{||| \mathbb{P}_0^{\perp}\nu_{\ell_{\max}}|||^*+ \frac{4\pi}{\kappa_0}||| \mathbb{P}_0^{\perp}\mathbb{Q}_{\ell_{\max}}\sigma_f|||^*} < \epsilon. \qquad\hphantom{+}
			\end{align*}
			
			Defining $\nu_{\ell_{\max}}^{\rm approx}:=\mathcal{B}\lambda_{R_{\epsilon}}^{\ell_{\max}}$ therefore completes the proof. 
		\end{proof}			
		
		Some explanatory remarks are now in order.
		
		\begin{remark}\label{rem:thm3}
			Consider the setting and proof of Theorem \ref{thm:3}. In practice, the function $\nu_{\ell_{\max}}^{\rm approx}:=\mathcal{B}\lambda_{R_{\epsilon}}^{\ell_{\max}}~\in W^{\ell_{\max}}(\partial \Omega) $ is represented as a vector $\boldsymbol{\nu}^{\rm approx} \in \mathbb{R}^{M+ N}$. This can be done as follows:
			
			First, let $\boldsymbol{\lambda}_{R_{\epsilon}} \in \mathbb{R}^M$, i.e., the $R_{\epsilon}^{\rm th}$ GMRES iterate, be  the vector representation of $\lambda_{R_{\epsilon}}^{\ell_{\max}} \in W^{\ell_{\max}}(\partial \Omega) $. Next, inspired by Definition \ref{def:Basis}, we equip the space $W^{\ell_{\max}}(\partial \Omega) $ with a basis of local spherical harmonics functions $\left\{\mathcal{Y}^i_{\ell m}\right\}$, ~$i\in \{1, \ldots, N\}$ and $\ell \in \{0, \ldots, \ell_{\max}\},~ -\ell \leq m \leq \ell$. Using the notation $[\boldsymbol{\lambda}_{R_{\epsilon}}^i]_{\ell}^m$, ~$i\in \{1, \ldots, N\}$ and $\ell \in \{1, \ldots, \ell_{\max}\},~ -\ell \leq m \leq \ell$ to denote the entries of $\boldsymbol{\lambda}_{R_{\epsilon}} $:
			
			\begin{itemize}
				\item We define the vector $\boldsymbol{\Psi} \in \mathbb{R}^{M+N}$ as
				\begin{align*}
					[\boldsymbol{\Psi}_{i}]_{\ell}^m:= \begin{cases}
						0 &\text{ for } i\in \{1, \ldots, N\}, \quad \ell, m =0,\\
						\frac{\kappa_0-\kappa_i}{\kappa_0}[\boldsymbol{\lambda}_{R_{\epsilon}}^i]_{\ell}^m \big(\text{DtN}\mathcal{Y}_{\ell m}^i, \mathcal{Y}_{\ell m}^i\big)&\text{ for } i\in \{1, \ldots, N\}, \quad \ell \in \{1, \ldots, \ell_{\max}\} \quad \text{and } \hspace{1mm} |m| \leq \ell.
					\end{cases}
				\end{align*}
				
				\item We define the vector $\boldsymbol{\sigma_f^Q} \in \mathbb{R}^{M+N}$ as
				\begin{align*}
					[\boldsymbol{\sigma_f^Q}_{i}]_{\ell}^m:= \frac{4\pi}{\kappa_0}\left(\sigma_f, \mathcal{Y}^i_{\ell m}\right)_{L^2(\partial \Omega_i)}, \qquad \hspace{0.9cm}\text{for } i\in \{1, \ldots, N\}, \quad \ell \in \{0, \ldots, \ell_{\max}\} \quad \text{and } \hspace{1mm} |m| \leq \ell,
				\end{align*}
			\end{itemize}
			
			It then follows from Definition \eqref{eq:thm31} of the affine map $\mathcal{B}$ that the vector representation $\boldsymbol{\nu}^{\rm approx} \in \mathbb{R}^{M+ N}$ of $\nu_{\ell_{\max}}^{\rm approx}$ is simply given by
			\begin{align}
				\boldsymbol{\nu}^{\rm approx}:=\boldsymbol{\Psi}+ \boldsymbol{\sigma_f^Q}.
			\end{align}
		\end{remark}
		
		\begin{remark}\label{rem:effect}
			Consider the proof of Theorem \ref{thm:3}. Equation \eqref{eq:iterations} describes the behaviour of our bound $R_{\epsilon}$ on the number of GMRES iterations required to obtain an approximate solution with relative error smaller than~$\epsilon$. In particular, we observe that
			\begin{itemize}
				\item $R_{\epsilon}$ grows moderately as $\log(\ell_{\max})$ for increasing $\ell_{\max}$. Here, $\ell_{\max}$ is the discretisation parameter for the approximation space. As discussed in Section \ref{sec:4} on numerical results, we typically pick $\ell_{\max} \in~\{5, \ldots, 20\}$ so we do not observe growth in the number of linear solver iterations for increasing $\ell_{\max}$ in practical numerical simulations.
				
				\item $R_{\epsilon}$ grows as moderately $\log(\Upsilon_{\rm GMRES})$ for increasing $\Upsilon_{\rm GMRES}$. Here, $\Upsilon_{\rm GMRES}$ is the constant defined in the proof of Theorem \ref{thm:3} and depends on geometrical parameters such as the radii of the spheres and the dielectric constants, and the continuity and inf-sup constant of the operator $\widetilde{\mathcal{A} }$. 
				
				\item $R_{\epsilon}$ grows as $\sqrt{\frac{C_{\widetilde{\mathcal{A}}}}{\alpha_0}}$. Here, $C_{\widetilde{\mathcal{A}}}$ and $\alpha_0$ are bounds on the largest and smallest eigenvalues respectively of the symmetric, finite-dimensional operator $\widetilde{\mathcal{A}}^{\rm sym}_{\ell_{\max}}$ (see Lemma \ref{lem:eigenvalues}) and are given by
				\begin{align*}
					\alpha_0 = \begin{cases}
						1 \quad &\text{if } \kappa > \kappa_0,\\
						\min \frac{\kappa}{\kappa_0} \quad &\text{if } \kappa < \kappa_0, \end{cases} \qquad C_{\widetilde{\mathcal{A} }} = 1 + \max\left\vert \frac{\kappa-\kappa_0}{\kappa_0}\right\vert \frac{ c_{\rm equiv} }{\sqrt{c_{\mathcal{V} } } }.				
				\end{align*}
			\end{itemize}

			Consequently, we would expect $R_{\epsilon}$ to be large if
			\begin{itemize}
				\item $\kappa < \kappa_0$ and $\min \frac{\kappa}{\kappa_0}$ is very small;
				\item $\kappa > \kappa_0$ and $\max \frac{\kappa}{\kappa_0}$ is very large;
				\item The coercivity constant $c_{\mathcal{V}}$ is very small. We have shown in the first paper \cite[Lemma 4.7]{Hassan} that $c_{\mathcal{V}} = \mathcal{O}(\delta)$ for small $\delta$, where $\delta$ is the minimum inter-sphere separation distance.
			\end{itemize}
		\end{remark}

		\begin{remark}\label{rem:errorvsresidual}
			Theorem \ref{thm:3} show that one can obtain an approximation to the solution $\nu_{\ell_{\max}}$ of the Galerkin discretisation \eqref{eq:Galerkina} up to a given \underline{relative error tolerance} $\epsilon$ by solving the linear system \eqref{eq:Matrix} using GMRES, and the number of iterations required is independent of $N$. Of course, in practice, the error of each GMRES iterate is unknown and the solver is typically run until some \underline{relative residual tolerance} is reached. Numerical experiments we have performed (see Section \ref{sec:4} for specific geometric settings) indicate that the relative residual is typically one order of magnitude larger than the relative error so a conservative strategy would be to set the GMRES tolerance two orders of magnitude lower than the desired relative error tolerance.
		\end{remark}
		
		\vspace{2mm}
		
		\noindent {\large \textbf{GMRES-based Solution Strategy for Obtaining the Induced Surface Charge}}~
		
		Given a free charge $\sigma_f \in {H}^{-\frac{1}{2}}(\partial \Omega)$, the goal is to obtain-- up to some given tolerance-- the solution~$\nu_{\ell_{\max}}~\in W^{\ell_{\max}}(\partial \Omega) $ to the Galerkin discretisation \eqref{eq:Galerkina} for the induced surface charge.
		
		\begin{enumerate}
			
			\item Fix $\ell_{\max} \in \mathbb{N}$ and use Definition \ref{def:Matrices} to compute the right-hand side vector $\boldsymbol{\sigma_f} \in \mathbb{R}^M$. Due to the use of the FMM, the total computational cost of this step is $\mathcal{O}(N)$.
			
			\item Use GMRES to solve-- up to some tolerance-- the linear system \eqref{eq:Matrix} involving $\boldsymbol{A}$ and $\boldsymbol{\sigma_f}$. This yields a solution vector $\boldsymbol{\lambda}^{\rm approx} \in \mathbb{R}^M$. Notice that $\boldsymbol{\lambda}^{\rm approx}$ is the vector representation of the function $\lambda_{\ell_{\max}}^{\rm approx} \in W_0^{\ell_{\max}}(\partial \Omega)$ which is an approximation to the true solution $\lambda_{\ell_{\max}}$ of the ``reduced'' Galerkin discretisation \eqref{eq:Galerkin}. Thanks to the FMM, the cost of a single matrix vector product involving $\boldsymbol{A}$ is $\mathcal{O}(N)$. Moreover, due to Theorem \ref{thm:3}, the total number of iterations of GMRES required in this step is independent of $N$. Consequently, the total computational cost of this step is also $\mathcal{O}(N)$.
			
			\item Following the procedure outlined in Remark \ref{rem:thm3}, use the solution vector $\boldsymbol{\lambda}^{\rm approx} \in \mathbb{R}^M$ to compute the vector $\boldsymbol{\nu}^{\rm approx} \in \mathbb{R}^{M+N}$. Clearly, the computational cost of this step is $\mathcal{O}(N)$.
			
			\item $\boldsymbol{\nu}^{\rm approx}$ is now the vector representation of some function $\nu^{\rm approx}_{\ell_{\max}} \in W^{\ell_{\max}}(\partial \Omega) $ that is the required approximation of the true solution $\nu_{\ell_{\max}}$ to the Galerkin discretisation \eqref{eq:Galerkina}. 
		\end{enumerate}

		We conclude this subsection by emphasising the key implication of Theorem \ref{thm:3} and our solution strategy: Given a geometrical configuration of $N$ particles that satisfies appropriate geometrical assumptions, we can compute- up to any given error tolerance-- the solution to the Galerkin discretisation \eqref{eq:Galerkin} using $\mathcal{O}(N)$ operations. In other words the numerical method is linear scaling in cost. Since the main result in \cite{Hassan} derived $N$-independent error estimates for the Galerkin discretisation \eqref{eq:Galerkin} and thus established $N$-error stability, we can conclude that the numerical method is also \emph{linear scaling in accuracy}, i.e., in order to obtain the approximate induced surface charge up to a fixed average or relative error, the computational cost scales linearly in $N$. 
		
		\vspace{2mm}
		\subsection{An Approach Based on the Conjugate-Gradient Method.}\label{sec:3.3}~
		
		\vspace{2mm}
		
		Notice that the convergence analysis we presented in Section \ref{sec:3.2} relied crucially on the similarity of the finite dimensional operators $\widetilde{\mathcal{A}}$ defined through Definition \ref{def:tildeA} and $\widetilde{\mathcal{A} }^{\rm sym}$ defined though Definition \ref{def:Asym}. The goal of this section is to further exploit this similarity and outline a solution strategy based on the use of the conjugate gradient (CG) method rather than GMRES. Throughout this subsection, we assume the setting of Section \ref{sec:3.2}.

		\begin{definition}\label{def:newRHS}
			Let $\ell_{\max} \in \mathbb{N}$, let $\sigma_f \in H^{-\frac{1}{2}}(\partial \Omega)$, and let the vector $\boldsymbol{\sigma_f} \in \mathbb{R}^M$ and the diagonal positive definite matrix $\boldsymbol{\rm DtN}^{\kappa} \in \mathbb{R}^{M \times M}$ be defined as in Definition \ref{def:Matrices}. Then we define the vector $\boldsymbol{\widetilde{\sigma_f}} \in \mathbb{R}^{M}$ as
			\begin{align*}
				\boldsymbol{\widetilde{\sigma_f}}:= \boldsymbol{\rm DtN}^{\kappa}\boldsymbol{\sigma_f}.
			\end{align*}
		\end{definition}
		
		Using Definition \ref{def:newRHS}, we can formulate a matrix equation associated with the symmetric finite-dimensional operator~$\widetilde{\mathcal{A} }^{\rm sym}$. \vspace{5mm}
		
		\noindent{\textbf{Symmetric Matrix Equation for $\widetilde{\mathcal{A} }^{\rm sym}$:}}~
		
		Let $\ell_{\max} \in \mathbb{N}$, let $\sigma_f \in H^{-\frac{1}{2}}(\partial \Omega)$, let the symmetric matrix $\boldsymbol{A}^{\rm sym} \in \mathbb{R}^{M \times M}$ be defined as in Definition \ref{def:Matrices}, and let the vector $\boldsymbol{\widetilde{\sigma_f}} \in \mathbb{R}^{M}$ be defined as in Definition \ref{def:newRHS}. Find a vector $\boldsymbol{\lambda}^{\rm sym} \in \mathbb{R}^{M}$ such that
		\begin{align}\label{eq:Matrix_sym}
			\boldsymbol{A}^{\rm sym}\boldsymbol{\lambda}^{\rm sym}= \boldsymbol{\widetilde{\sigma_f}}.
		\end{align}
		
		Notice that Equation \eqref{eq:Matrix_sym} is well-posed since $\boldsymbol{A}^{\rm sym}$ is symmetric positive definite (see Lemma \ref{lem:eigenvalues}). Furthermore, if we denote by $\boldsymbol{\lambda} \in \mathbb{R}^M$ and $\boldsymbol{\lambda}^{\rm sym} \in \mathbb{R}^{M}$ the solutions to the matrix equations \eqref{eq:Matrix} and \eqref{eq:Matrix_sym} respectively, then it is easy to see that
		\begin{align}\label{eq:sym}
			\boldsymbol{\lambda}= (\boldsymbol{\rm DtN}^{\kappa})^{-1} \boldsymbol{\lambda}^{\rm sym}.
		\end{align}
		
		Equation \eqref{eq:sym} suggests that we can avoid the use of GMRES for solving the non-symmetric matrix Equation \eqref{eq:Matrix} and instead solve the symmetric matrix Equation \eqref{eq:Matrix_sym} using the CG method. Our next result pertains to the convergence analysis of this alternative approach.

		\begin{theorem}[Convergence Analysis for CG-based Strategy]\label{thm:4}~\\
			Let $\ell_{\max} \in \mathbb{N}$, let $\sigma_f \in H^{-\frac{1}{2}}(\partial \Omega)$, let $\nu_{\ell_{\max}} \in W^{\ell_{\max}}(\partial \Omega) $ be the unique solution to the Galerkin discretisation defined through Equation \eqref{eq:Galerkina} with right hand side given by $\sigma_f$, and let $\mathbb{P}_0^{\perp} \colon H^{\frac{1}{2}}(\partial \Omega) \rightarrow \breve{H}^{\frac{1}{2}}(\partial \Omega)$ and $\mathbb{Q}_{\ell_{\max}}\colon H^{\frac{1}{2}}(\partial \Omega) \rightarrow  W^{\ell_{\max}}(\partial \Omega) $ be the orthogonal projection operators. Then for every $\epsilon > 0$ there exists a function $\nu_{\ell_{\max}}^{\rm{approx}} \in W^{\ell_{\max}}(\partial \Omega) $ and a natural number $S_{\epsilon}>0$ that does not depend on the number of dielectric spheres $N$ such that $\nu_{\ell_{\max}}^{\rm{approx}} $ can be computed using at most $S_{\epsilon}$ iterations of the conjugate gradient method and such that the following error estimate holds
			\begin{align*}
				\frac{||| \nu_{\ell_{\max}}^{\rm{approx}}  - \nu_{\ell_{\max}}|||^*}{||| \mathbb{P}_0^{\perp}\nu_{\ell_{\max}}|||^*+ \frac{4\pi}{\kappa_0}||| \mathbb{P}_0^{\perp}\mathbb{Q}_{\ell_{\max}}\sigma_f|||^*}\,  < \epsilon.
			\end{align*}
		\end{theorem}
		\begin{proof}	
			The proof of Theorem \ref{thm:4} is very similar to the proof of Theorem \ref{thm:3} so we omit it for the sake of brevity. A detailed proof can be found in \cite{Hassan_Dis}. Let us remark however that the bound on the number of CG iterations $S_{\epsilon}>0$ that we obtain is very similar to the bound $R_{\epsilon}$ for GMRES.

		\end{proof}

		\noindent {\large \textbf{CG-based Solution Strategy for Obtaining the Induced Surface Charge}}~
		
		Given a free charge $\sigma_f \in {H}^{-\frac{1}{2}}(\partial \Omega)$, the goal is to obtain-- up to some given tolerance-- the solution~$\nu_{\ell_{\max}} \in$ $W^{\ell_{\max}}(\partial \Omega) $ to the Galerkin discretisation \eqref{eq:Galerkina} for the induced surface charge.
		
		\begin{enumerate}
			
			\item Fix $\ell_{\max} \in \mathbb{N}$ and use Definitions \ref{def:Matrices} and \ref{def:newRHS} to compute the right-hand side vector $\boldsymbol{\widetilde{\sigma_f}} \in \mathbb{R}^M$. Due to the use of the FMM and the fact that $\boldsymbol{\rm DtN}^{\kappa}$ is a diagonal matrix, the total computational cost of this step is $\mathcal{O}(N)$.
			
			\item Use the CG method to solve-- up to some tolerance-- the linear system \eqref{eq:Matrix_sym} involving $\boldsymbol{A}^{\rm sym}$ and $\boldsymbol{\widetilde{\sigma_f}}$. This yields a solution vector $\boldsymbol{\lambda}^{\rm sym}_{\rm approx} \in \mathbb{R}^M$. Once again, the cost of a single matrix vector product involving $\boldsymbol{A}^{\rm sym}$ is $\mathcal{O}(N)$ thanks to the FMM. Due to Theorem \ref{thm:4}, the total number of CG iterations required in this step is independent of $N$, and therefore the total computational cost of this step is also~$\mathcal{O}(N)$.
			
			\item Use the solution vector $\boldsymbol{\lambda}^{\rm sym}_{\rm approx} \in \mathbb{R}^M$ to compute the vector $\boldsymbol{\nu}^{\rm approx} \in \mathbb{R}^{M+N}$. This can be done by recalling Equation \eqref{eq:sym} and following the procedure outlined in Remark \ref{rem:thm3} with a few obvious modifications. The computational cost of this step is $\mathcal{O}(N)$.
			
			\item $\boldsymbol{\nu}^{\rm approx}$ is now the vector representation of some function $\nu^{\rm approx}_{\ell_{\max}} \in W^{\ell_{\max}}(\partial \Omega) $ that is the required approximation of the true solution $\nu_{\ell_{\max}}$ to the Galerkin discretisation \eqref{eq:Galerkina}. 
		\end{enumerate}

		\section{Numerical Experiments}\label{sec:4}
		Throughout this section, we assume the setting described in Sections \ref{sec:2} and \ref{sec:3}. Our goal is now two-fold. First, we wish to present numerical results supporting the conclusions of Theorem \ref{thm:3} and \ref{thm:4} concerning the number of linear solver iterations required to obtain an approximate solutions up to a given and fixed relative error. Second, we would like to provide numerical evidence that that the solution strategies presented in Section~\ref{sec:3} are indeed linear scaling in accuracy. Demonstrating this linear scaling behaviour is rather subtle as we now explain.
		

		
		The key complication is that the FMM, which is used to compute matrix-vector products involving the solution matrix, is not exact (see, e.g., \cite{greengard1}) and introduces an approximation error as soon as one increases the depth of the octree structure of the bounding box containing all multipole sources. This is because the FMM uses an approximate so-called `far field' to compute interactions between multipole sources that belong to well-separated leaves of the octree. Note that increasing the depth of the octree is required to maintain linear scaling complexity for larger matrices, i.e., for matrices corresponding to systems with an increasing number $N$ of dielectric particles. 
		In principle, the error introduced by the far-field computations can be made arbitrarily small by increasing the maximal degree of spherical harmonics that are used in the multipole expansions of the underlying kernel but increasing the expansion degree also increases the computational cost of each matrix-vector product. There is thus a tradeoff between the computational cost and accuracy of the FMM. 
		
		Consequently, in Subsection \ref{sec:4.2}, we first explore the interplay between the FMM error and discretisation error for different values of the system parameters. Based on these results, we pick FMM system parameters that result in a linear scaling in accuracy solution strategy such that the FMM error does not dominate the discretisation error.
		
		\vspace{2mm}
		
		\subsection{Validation of Theoretical Results}\label{sec:4.1}~
		
		\vspace{2mm}
		
		Since standard FMM libraries typically consider point charges as input, we have used instead a modification of the ScalFMM library (see \cite{lindgren2018} for an explanation of the modifications and \cite{agullo2014task, blanchard2015scalfmm, messner2012optimized} for details on ScalFMM). The subsequent numerical simulations were performed with a single level FMM octree so as to avoid using the approximate `far-field' and to perform all computations using the exact FMM `near-field' instead. Unless stated otherwise, the discretisation parameter was fixed as $\ell_{\max}=5$.\vspace{2mm}
		
		\noindent {\textbf{$N$-independence}}~
		Our first set of numerical experiments is designed to demonstrate that the number of linear solver iterations required to obtain an approximation up to a given relative error to $\nu_{\ell_{\max}}$ is independent of the number $N$ of dielectric particles. 
		
		We adopt the following geometric setting: We consider dielectric spheres of radius 1 and dielectric constant $\kappa=10$ arranged on a regular cubic lattice of edge length $2.5$. The spheres carry alternating unit positive and negative charges, and the background medium is assumed to be vacuum so that $\kappa_0=1$. An example of the problem geometry is displayed in Figure \ref{fig:11}. The number of spheres is increased simply by increasing the size of the lattice.
		
		\begin{figure}[h]
			\centering
			\begin{subfigure}[t]{0.48\textwidth}
				\centering
				\includegraphics[width=\textwidth]{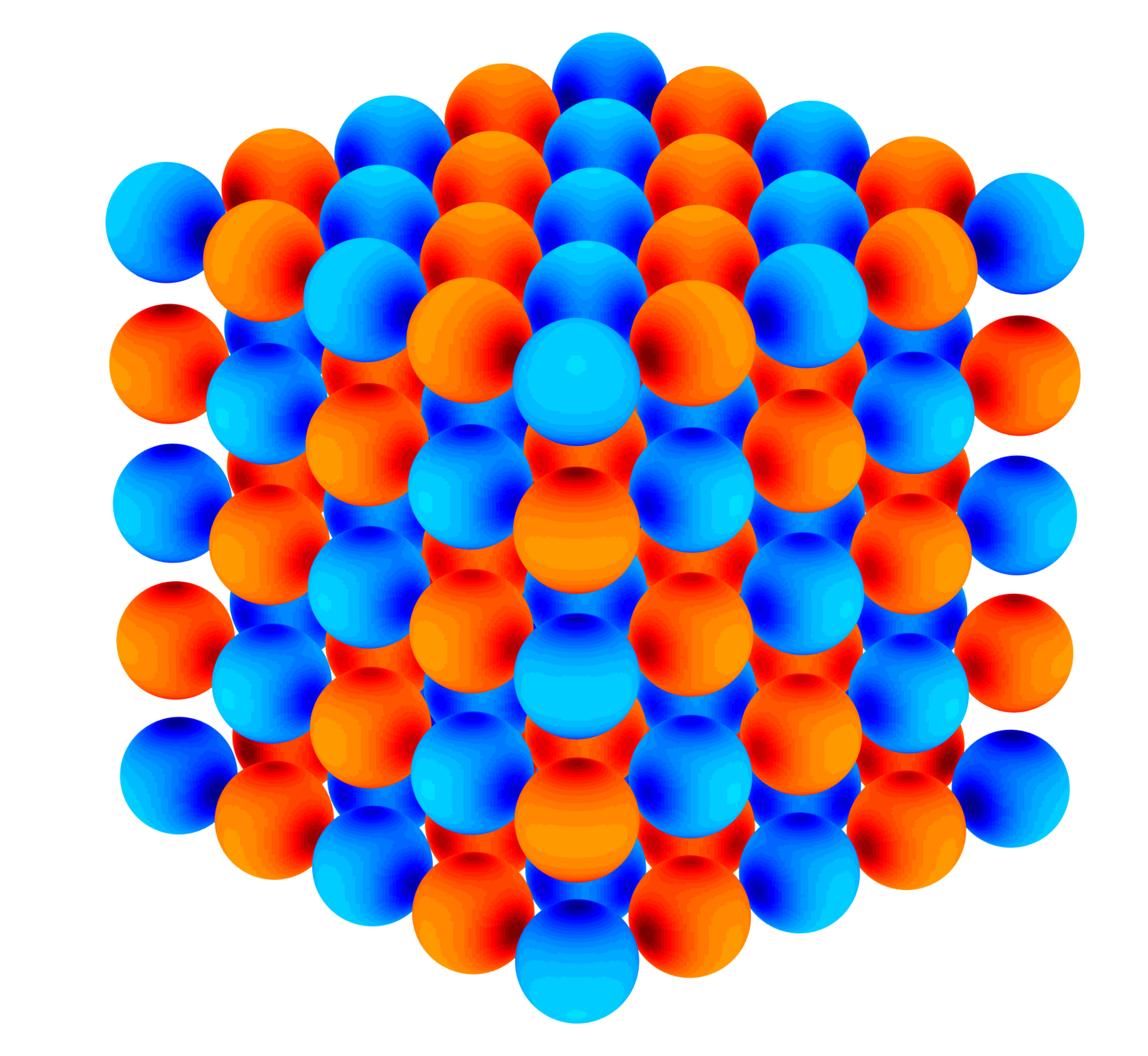} 
				\caption{The basic geometric setting of our numerical experiments. The dielectric spheres are arranged on a three dimensional, regular cubic lattice with edge length $2.5$.}
				\label{fig:11}
			\end{subfigure}\hfill
			\begin{subfigure}[t]{0.48\textwidth}
				\centering
				\includegraphics[width=\textwidth]{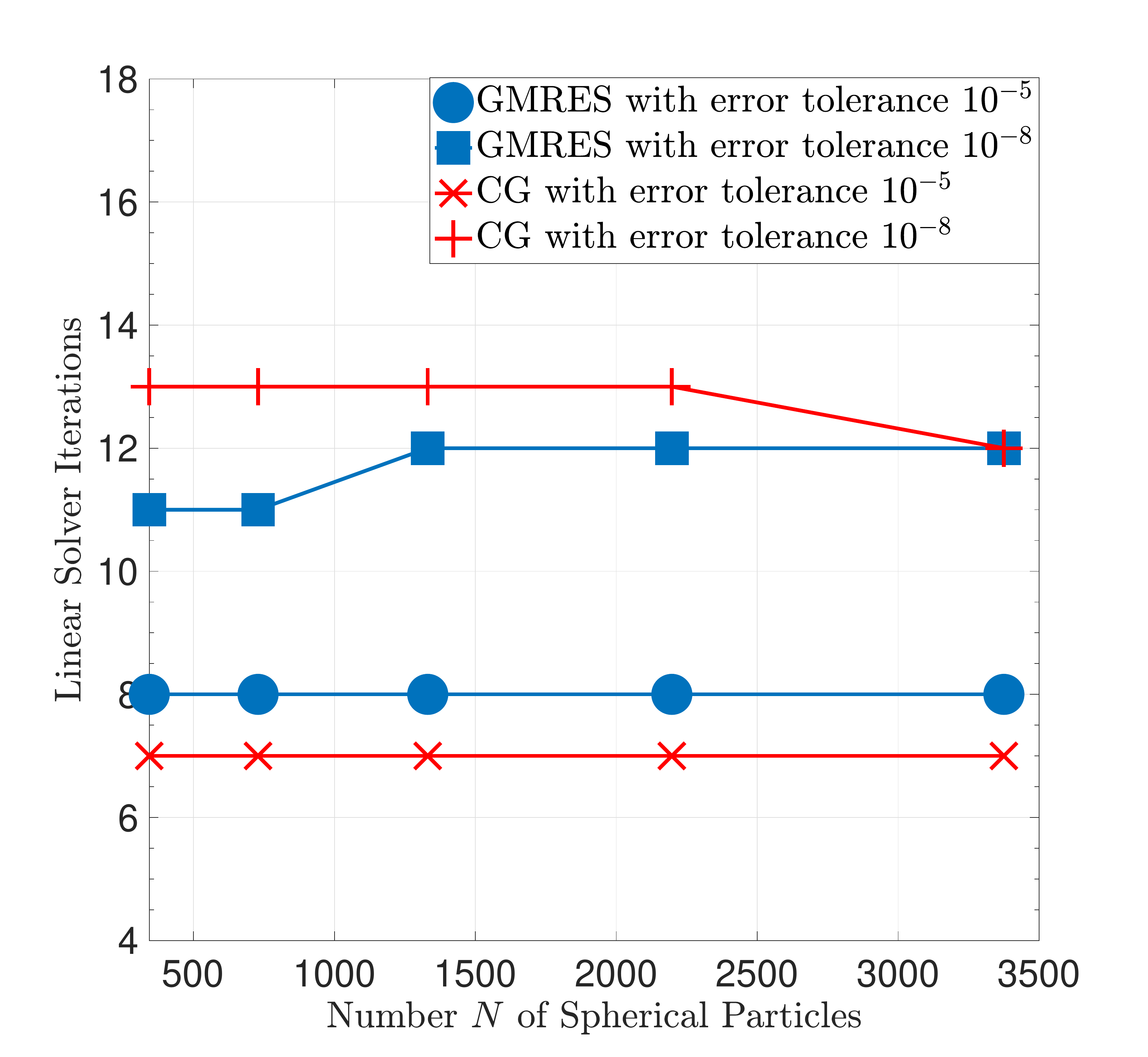} 
				\caption{Linear Solver iterations required to obtain an approximate solution with a given error tolerance as a function of the number $N$ of dielectric particles.}
				\label{fig:12}
			\end{subfigure}
			\caption{Left: An example of the problem geometry. Right: The linear solver iterations as a function of the number $N$ of spherical particles.}
		\end{figure}
		
		Figure \ref{fig:12} displays the number of GMRES and CG iterations required to produce an approximation $\nu_{\ell_{\max}}^{\rm approx}$ of the true solution $\nu_{\ell_{\max}}$ to the Galerkin discretisation \eqref{eq:Galerkina} with a given relative error tolerance. The true solution $\nu_{\ell_{\max}}$ was calculated by solving the linear system with tolerance $10^{-13}$ and the relative error was calculated exactly as in Theorems \ref{thm:3} and \ref{thm:4}. Clearly, the numerical results agree with Theorems \ref{thm:3} and \ref{thm:4}. Interestingly, for lower error tolerances, the number of CG iterations is smaller than the number of GMRES iterations and the situation is reversed for higher tolerances.\vspace{2mm}

		\noindent {\textbf{Dependence on the Dielectric Constant Ratio}}~
		Next, we explore the effects of different dielectric constant ratios on the number of linear solver iterations required to obtain an approximate solution satisfying a given error tolerance. The bounds obtained in Theorems \ref{thm:3} and \ref{thm:4} and Remark \ref{rem:effect} indicate that
		\begin{itemize}
			\item If $\kappa> \kappa_0$, then the number of iterations should grow at most as $ \sqrt{\frac{\max \kappa}{\kappa_0}}$ for increasing $\frac{\max\kappa}{\kappa_0}$.
			
			\item If $\kappa< \kappa_0$, then the number of iterations should grow at most as $ \sqrt{\frac{\kappa_0}{ \min\kappa}}$ for increasing $\frac{\kappa}{\min\kappa_0}$.
		\end{itemize}
		
		The problem geometry is similar to the one from the previous test case. We consider a total of 125 identical dielectric spheres of radius 1 with alternating positive and negative charge, arranged on a regular cubic lattice of edge length $2.5$. We set $\kappa_0=1$ and we allow the dielectric constant $\kappa$ of the spheres to vary from extremely high to extremely low. In all cases, the true solution $\nu_{\ell_{\max}}$ was calculated by solving the linear system with tolerance~$10^{-13}$.
		
		\begin{figure}[h!]
			\centering
			\begin{subfigure}[t]{0.48\textwidth}
				\centering
				\includegraphics[width=1\textwidth]{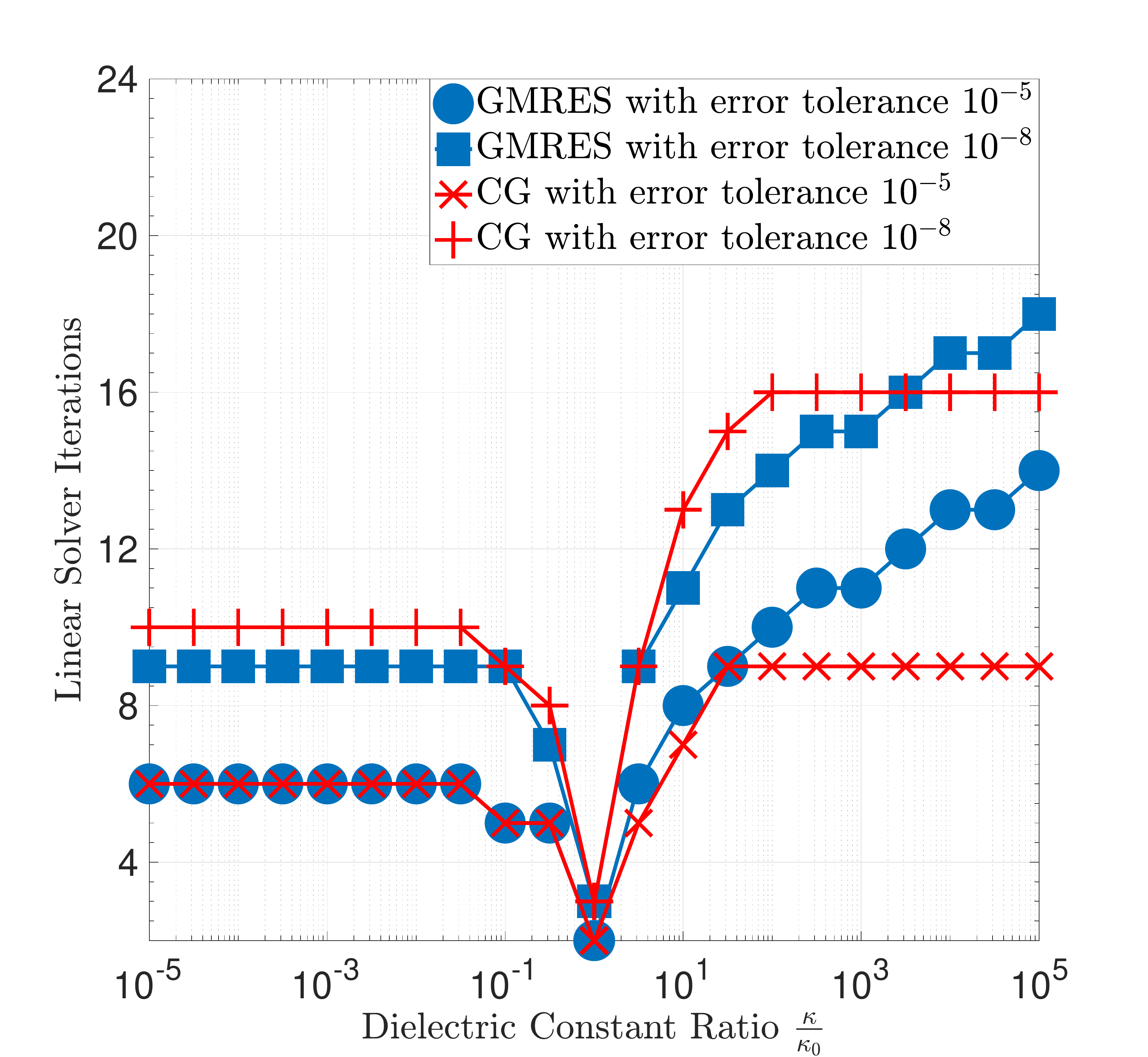} 
				\caption{Linear Solver iterations required to obtain an approximate solution with a given error tolerance as a function of the dielectric constant ratio.}
				\label{fig:21}
			\end{subfigure}\hfill
			\begin{subfigure}[t]{0.48\textwidth}
				\centering
				\includegraphics[width=1\textwidth]{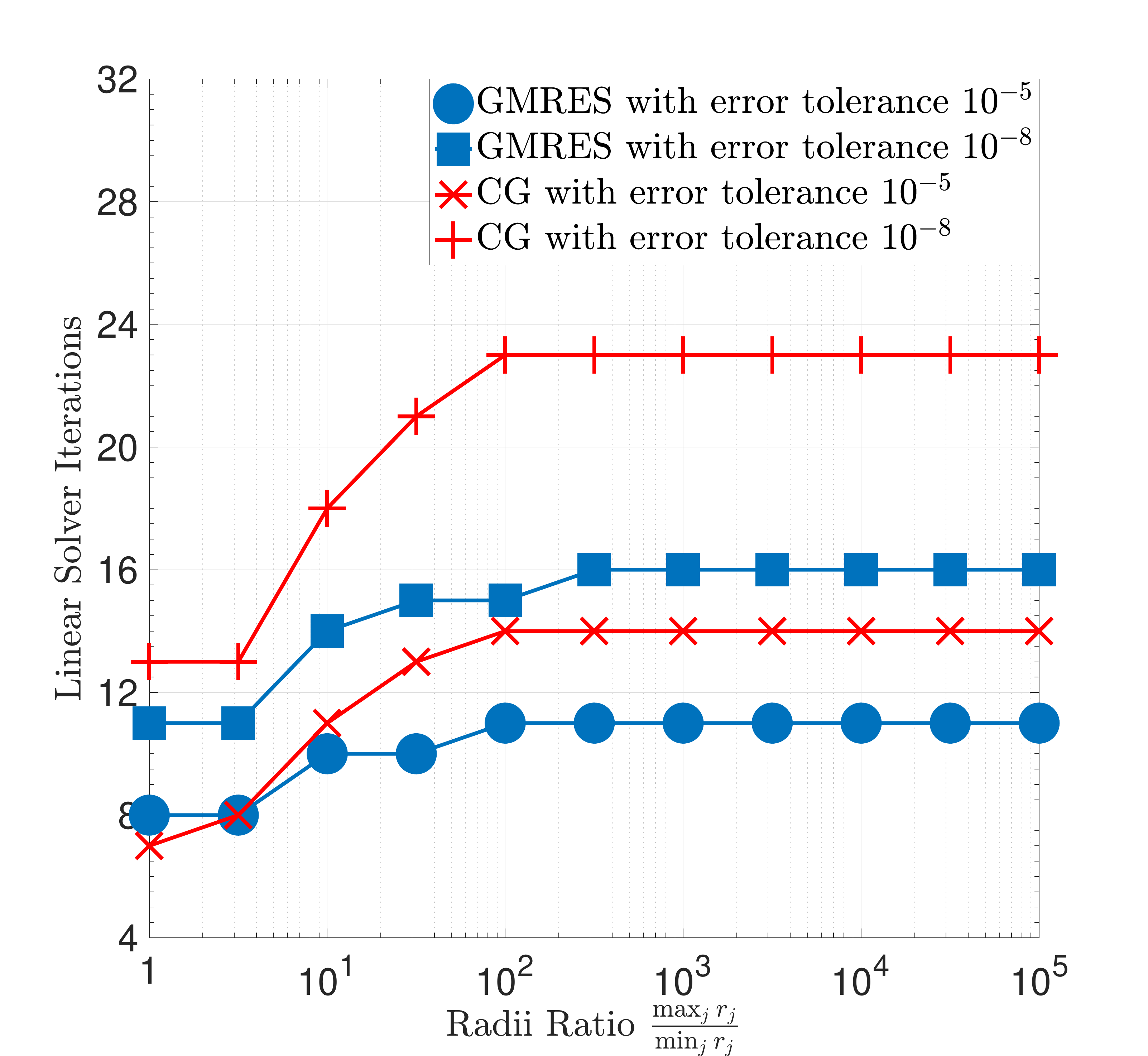} 
				\caption{Linear Solver iterations required to obtain an approximate solution with a given error tolerance as a function of the radii ratio.}
				\label{fig:31}
			\end{subfigure}
			\caption{Left: The linear solver iterations as a function of the dielectric constant ratio $\frac{\kappa}{\kappa_0}$. Right: The linear solver iterations as a function of the radii ratio $\frac{\max r_j}{\min r_j}$.}
		\end{figure}
		
		Figure \ref{fig:21} displays the number of GMRES and CG iterations required to produce an approximation to the true solution satisfying a given error tolerance. For very large dielectric ratio $\frac{\kappa}{\kappa_0}$, the number of GMRES iterations seems to grow logarithmically while the number of CG iterations grows at first but soon reaches a plateau. On the other hand, for very small dielectric ratio $\frac{\kappa}{\kappa_0}$, the number of iterations in both cases quickly reach a plateau. These results suggest that the bounds we have obtained in Theorems \ref{thm:3} and \ref{thm:4} may not be sharp. Interestingly, we again observe that for low error tolerances, CG outperforms GMRES.\vspace{2mm}
		
		\noindent{\textbf{Dependence on the Radii Ratio}}~		
		We now consider the dependence of the number of linear solver iterations on the ratio of the maximum and minimum radius of the dielectric spherical particles. As mentioned in Remark \ref{rem:effect}, we expect the number of iterations to grow at most as $\log(\frac{\max_{j=1, \ldots, N} r_j}{\min_{j=1, \ldots, N} r_j})$.
		
		The geometric setting consists of 125 dielectric spheres with dielectric constant $\kappa=10$, carrying alternating unit positive and negative charges, arranged on a regular cubic lattice of edge length $2.5$. We set $\kappa_0=1$ and we further set the radii of half of the dielectric spheres to one. The radii of all other dielectric spheres is set to $r$ and we vary $r$ from $10^{-5}$ to 1. As before, the true solution is calculated by setting the linear solver tolerance to~$10^{-13}$.
		
		Figure \ref{fig:31} displays the numerical results. In contrast to our theoretical results, the numerical simulations suggest that the number of iterations at first grow logarithmically as the radii ratio $\frac{\max_{j=1, \ldots, N} r_j}{\min_{j=1, \ldots, N} r_j}$ increases but the growth soon stops and for sufficiently large radii ratio, the number of iterations remains constant. We observe that for large radii ratios, GMRES significantly outperforms CG.\vspace{2mm}

		{\textbf{Dependence on the Separation Distance}}~
		Finally, we explore the dependence of the number of linear solver iterations on the minimal inter-sphere separation distance. We recall from Remark \ref{rem:effect} that we expect the number of linear solver iterations to grow at most as $c^{-\frac{1}{4}}_{\mathcal{V}}$ where $c_{\mathcal{V}}$ is the coercivity constant of the single layer boundary operator. Moreover, we have shown in the first article \cite[Lemma 4.7]{Hassan} that $c_{\mathcal{V}}= \mathcal{O}(\delta)$ for small $\delta$ where $\delta$ is the minimum inter-sphere separation distance.
		
		We consider once again 125 identical dielectric spheres of radius 1 and dielectric constant $\kappa=10$ with alternating positive and negative charge, arranged on a regular cubic lattice of edge length $E$. We assume the background medium to be vacuum so that $\kappa_0=1$ and we vary the edge length $E$ from $2+10^{-4}$ to 7. Thus, the minimum separation varies from $10^{-4}$ to $5$. In all cases, the true solution $\nu_{\ell_{\max}}$ was calculated by solving the linear system with tolerance~$10^{-13}$. Figure \ref{fig:41} displays the numerical results.
		
		\begin{figure}[h]
			\centering
			\begin{subfigure}[t]{0.48\textwidth}
				\centering
				\includegraphics[width=1\textwidth]{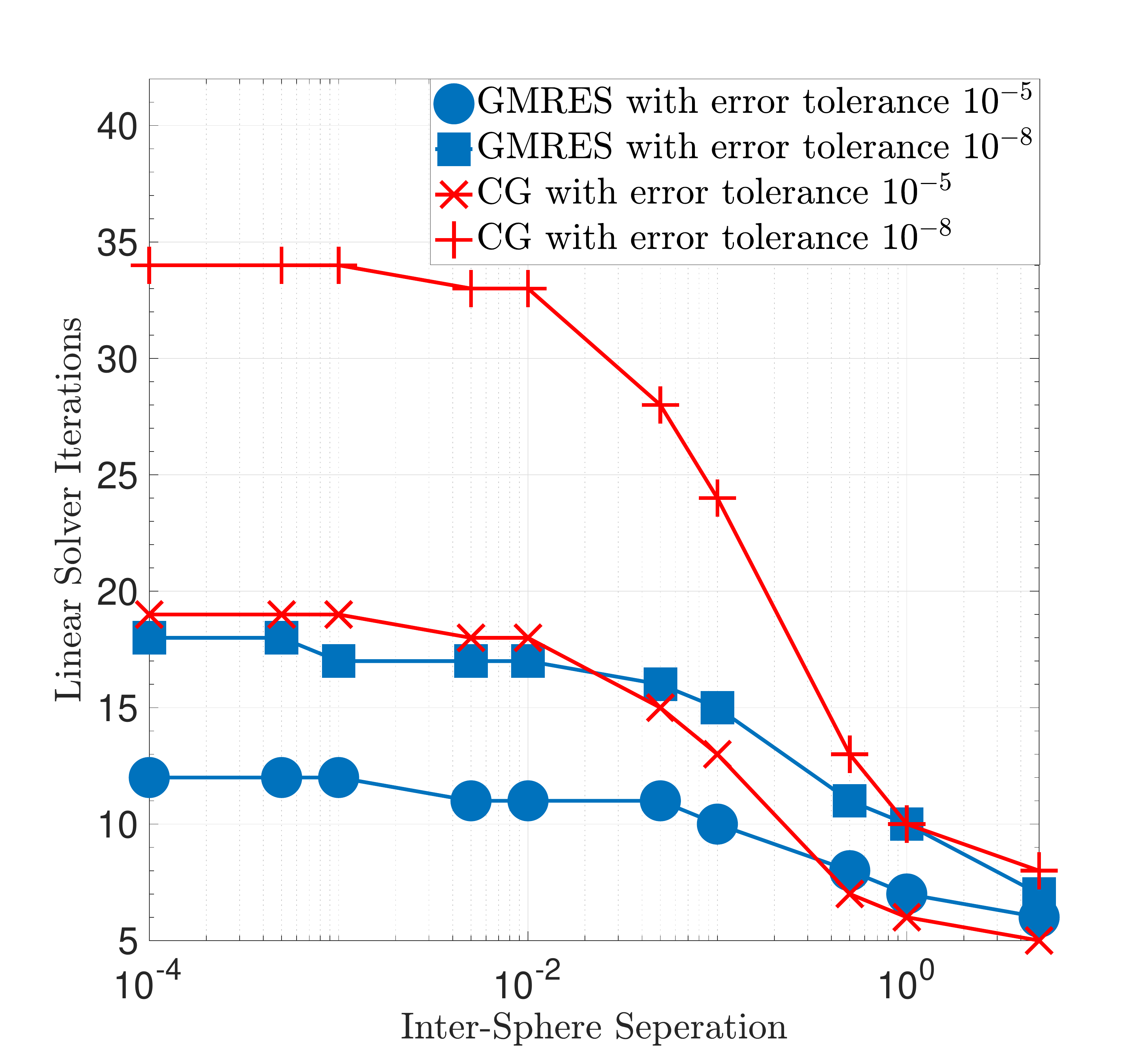} 
				\caption{Linear Solver iterations required to obtain an approximate solution with a given error tolerance as a function of the minimum inter-sphere separation distance.}
				\label{fig:41}
			\end{subfigure}\hfill
			\begin{subfigure}[t]{0.48\textwidth}
				\centering
				\includegraphics[width=1\textwidth]{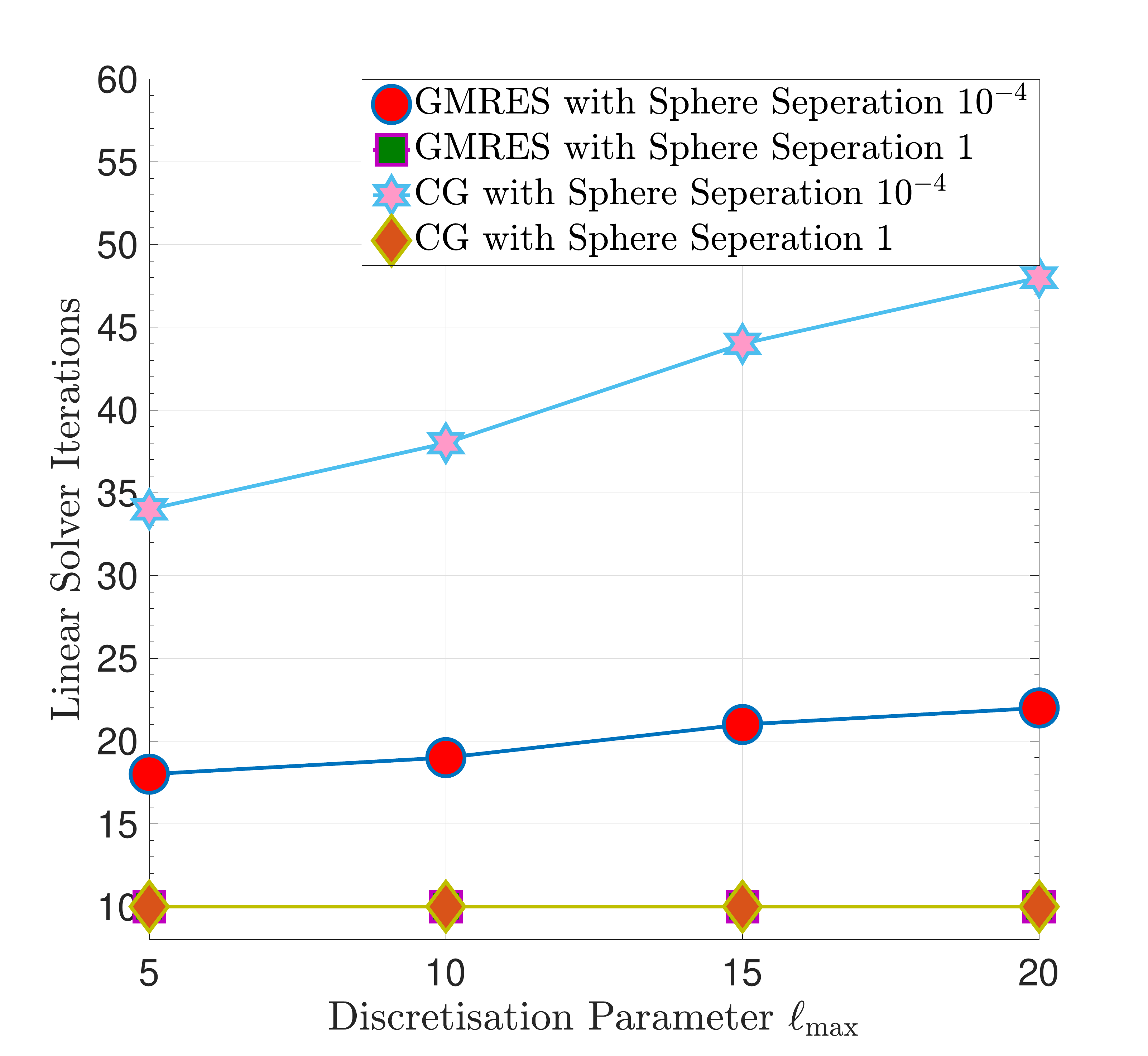} 
				\caption{Linear Solver iterations required to obtain an approximate solution with error tolerance $10^{-8}$ as a function of $\ell_{\max}$ for very small and moderate separation distances.}
				\label{fig:42}
			\end{subfigure}
			\caption{Left: The linear solver iterations as a function of the minimum sphere separation. Right: The linear solver iterations as a function of $\ell_{\max}$ for very small and moderate separations.}
		\end{figure}

		There are two features of interest in these numerical results. First, we observe that for very small separation distances, the number of CG iterations far exceeds the number of GMRES iterations. Second, we observe that while the number of iterations in both cases grows as the separation distance decreases, the growth stops at some point and the number of iterations plateaus. We conjecture that this is due to the fact that we use the continuity constant $C_{\widetilde{\mathcal{A} }}$ of the infinite-dimensional operator $\widetilde{\mathcal{A}}$ to bound the largest eigenvalue of the solution matrix, and functions which achieve (or approximately achieve) the upper bound $C_{\widetilde{A}}$ are not well-approximated in the approximation space $W^{\ell_{\max}}_0(\partial \Omega) $ for small $\ell_{\max}$.

		To test the above hypothesis, we plot in Figure \ref{fig:42} the number of linear solver iterations for different values of $\ell_{\max}$ with edge lengths $E=2+10^{-4}$ and $E=3$. The error tolerance was set to $10^{-8}$. We observe that the number of iterations remains constant in the case $E=1$ but increases in the case $E=10^{-4}$, which supports our conjecture. Let us remark here that a possible strategy for the treatment of point singularities that arise due to small separation distances between the particles has, for instance, been proposed in the contribution \cite{MR3493124}. The authors in \cite{MR3493124} derive analytical expressions for the induced potential both inside and outside a dielectric spherical particle due to a general multipole source using the method of image charges and image potentials. These analytical expressions are then combined with the classical method of moments to construct a hybrid algorithm. Numerical experiments indicate that the hybrid method has significantly better accuracy than the classical method of moments and also leads to solution matrices that do not suffer from ill-conditioning.

		\vspace{2mm}
		
		\subsection{FMM Error and Linear Scaling Solution Strategy}\label{sec:4.2}~
		
		\vspace{2mm}
		
		We now explore the interplay between the discretisation error and the error introduced by the FMM. As mentioned at the beginning of this section, given a system of $N$ interacting particles, the FMM can compute matrix-vector products using only $\mathcal{O}(N)$ operations but this comes at the cost of introducing an approximation error. The approximation error typically grows as one increases the tree depth $D$, i.e., the number of levels in the octree structure of the FMM bounding box because this usually results in more particle interactions being computed using the less-accurate 'far-field' computation of the FMM. Conversely, the approximation error decreases if one increases the maximal degree $P$ of spherical harmonics used in the multipole expansion of the FMM kernel. 
		
		Consequently, the first goal of this section is to observe numerically how the FMM error compares with the discretisation error for different values of $D, P$ and $\ell_{\max}$. Based on these results, we propose appropriate values of $D$ and $P$ such that for an increasing number $N$ of particles, an approximate solution to the Galerkin equation \eqref{eq:Galerkina} can be computed in $\mathcal{O}(N)$ operations and such that the FMM approximation error does not dominate the discretisation error. Finally, we present numerical results on the computation times of our algorithm for increasing $N$ which utilise the proposed values of $P$ and $D$.
		

		All subsequent numerical experiments involve the following geometric setting: We consider two types of dielectric spheres arranged on a regular cubic lattice of edge length $7$. The first type of dielectric spheres have radius 3, dielectric constant $10$ and carry unit negative charge, and the second type of dielectric spheres have radius 2, dielectric constant $5$ and carry unit positive charge. The background medium is assumed to be vacuum so that $\kappa_0=1$, and the number of spheres is increased simply by increasing the size of the lattice.

		We consider two choices of the discretisation parameter, namely, $\ell_{\max}=5$ and $\ell_{\max}=10$. We compute so-called `pure discrete' solutions to the Galerkin discretisation \eqref{eq:Galerkina} in each of the two cases for a different number $N$ of spherical particles. These solutions are all computed using a one-level FMM tree, which results in the use of the exact `near-field' computation of the FMM. Additionally, the linear solver tolerance in each case is set to $10^{-10}$. Consequently, these pure discrete solutions can be assumed to have negligible FMM and linear solver errors. Unfortunately, since we wish to use the exact `near-field' computations of the FMM, the computational cost of obtaining each pure discrete solution grows as $\mathcal{O}(N^2)$. This limits the total number of spheres we consider to~$1728$.

		Next, we compute `approximate' solutions for both values of $\ell_{\max}$ by repeating the above computations for different values of the FMM parameters $D$ and $P$ whilst keeping all other parameters identical. In addition, we compute the `reference' solution $\nu$ to the BIE \eqref{eq:3.3a} by setting $\ell_{\max}=20$ and using a linear solver tolerance of $10^{-13}$. Our goal now is to
		
		\begin{itemize}
			\item Compute the \emph{approximation error due to the FMM} by comparing the pure discrete solutions and the approximate solutions computed above;
			\item Compute the \emph{discretisation error } by comparing the pure discrete solutions and the reference solution computed above.
		\end{itemize}

		Figures \ref{fig:71} and \ref{fig:72} display the relative FMM and discretisation errors (relative errors are calculated using the $||| \cdot |||^*$ norm) in all cases. Although, it is difficult to form a definitive conclusion based on a limited data set, we observe two broad trends:

		\begin{itemize}
			\item The FMM error far exceeds the discretisation error if $N$ is small and $D$ is high. We deduce from this that $D$ should only be increased in proportion with $N$. Based on these results, the FMM error is minimised if a 2 level tree is used for $512$ particles, which translates to exactly 8 particles per leaf. One possible strategy for attaining this optimal particle per leaf ratio is to to start with a 1 level tree and increase $D$ until there are no more than 32 particles in each leaf. Notice that the choice of $D$ is independent of $\ell_{\max}$ and depends only on the number of particles $N$.
			
			\item As expected, the FMM error decreases as the expansion degree $P$ is increased. Unfortunately, the computational cost of the FMM grows as $\mathcal{O}(P^3)$ so the optimal strategy is to find the minimum $P$ such that the FMM error is strictly smaller than the discretisation error. Based on the results presented here, one possible choice could be $P = 2\ell_{\max}$. We remark that the choice of $P$ is independent of the number of particles $N$ and depends only on the discretisation parameter $\ell_{\max}$.
		\end{itemize}

		\begin{figure}[h]
			\centering
			\begin{subfigure}[t]{0.48\textwidth}
				\centering
				\includegraphics[width=1\textwidth]{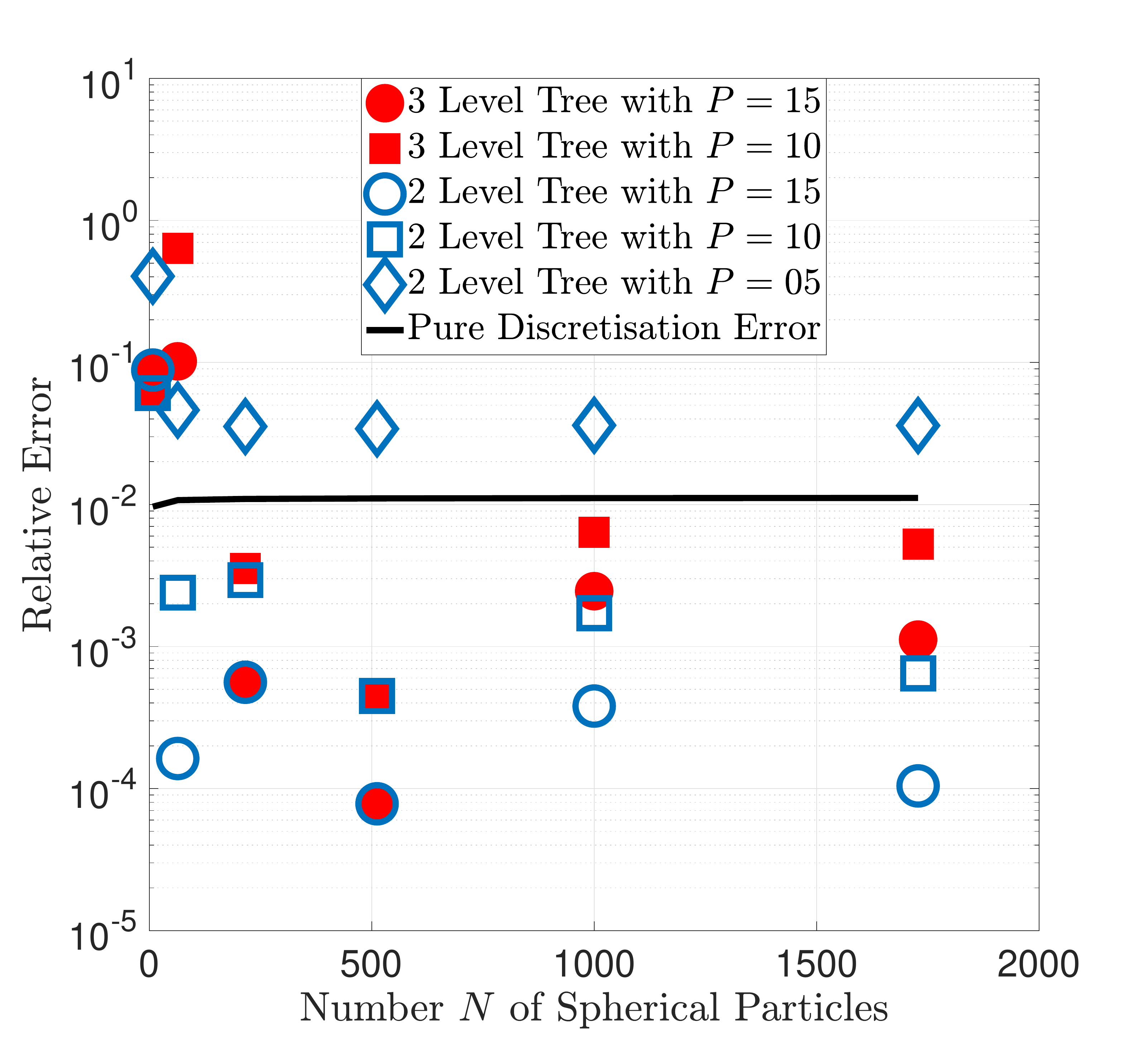} 
				\caption{Relative FMM and discretisation errors as a function of the number $N$ of spherical particles in the case $\ell_{\max}=5$.}
				\label{fig:71}
			\end{subfigure}\hfill
			\begin{subfigure}[t]{0.48\textwidth}
				\centering
				\includegraphics[width=1\textwidth]{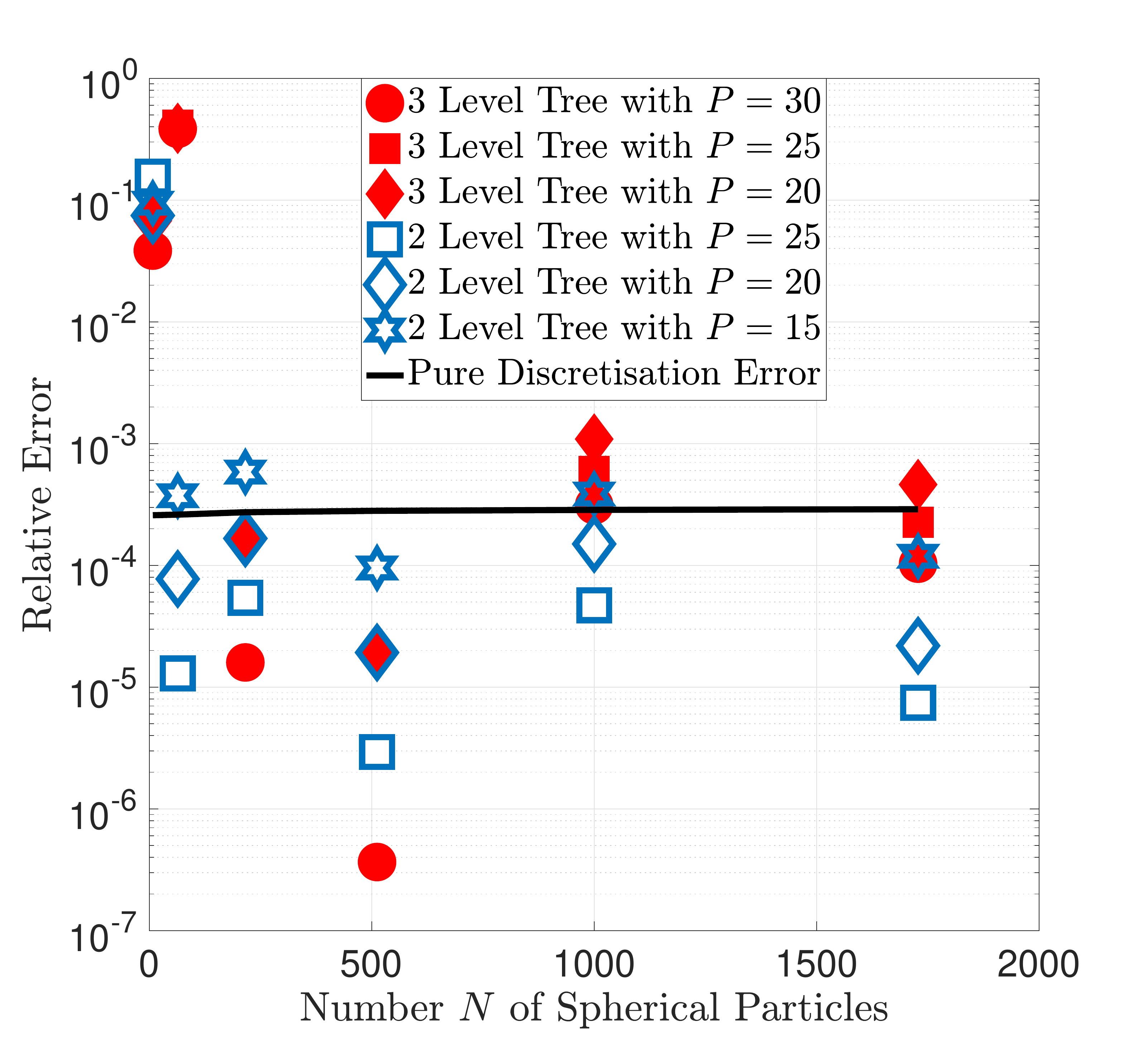} 
				\caption{Relative FMM and discretisation errors as a function of the number $N$ of spherical particles in the case $\ell_{\max}=10$.}
				\label{fig:72}
			\end{subfigure}
			\caption{FMM errors vs. discretisation errors for different choices of FMM parameters.}
		\end{figure}
		
		\begin{figure}[h]
			\centering
			\begin{subfigure}[t]{0.48\textwidth}
				\centering
				\includegraphics[width=\textwidth]{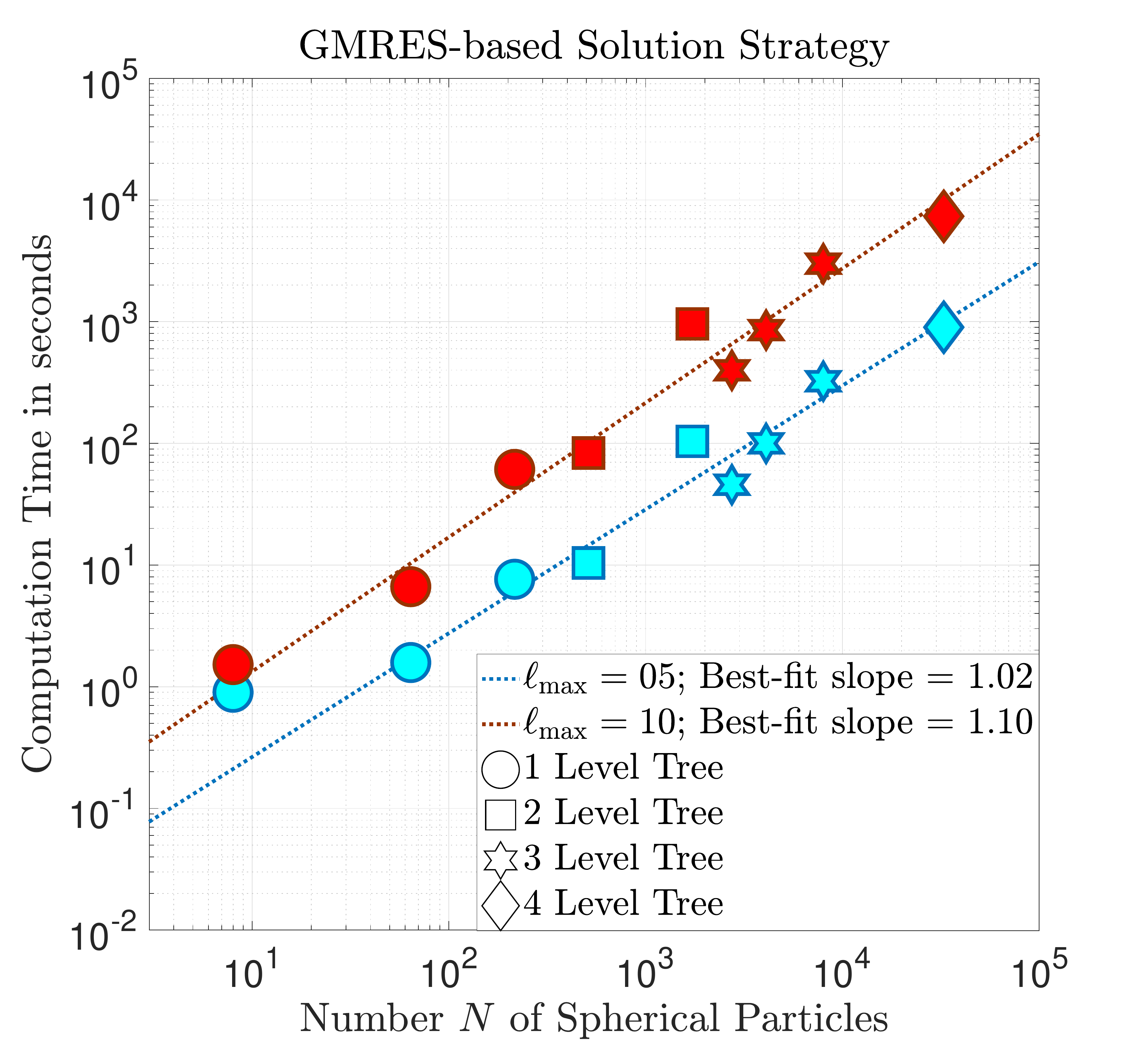} 
				\caption{Computation times for the GMRES-based solution strategy as a function of the number $N$ of spherical particles.}
				\label{fig:81}
			\end{subfigure}\hfill
			\begin{subfigure}[t]{0.48\textwidth}
				\centering
				\includegraphics[width=\textwidth]{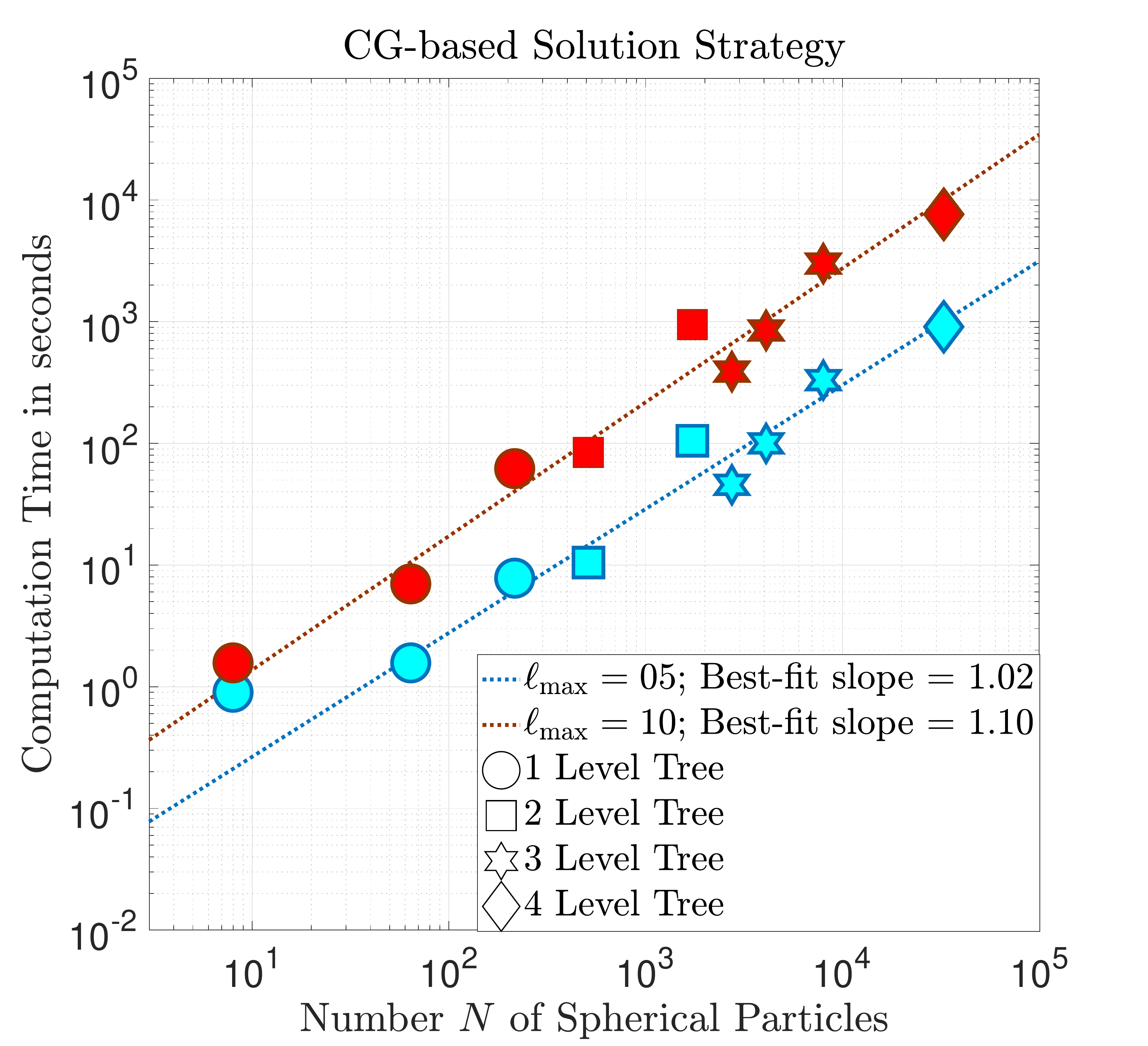} 
				\caption{Computation times for the CG-based solution strategy as a function of the number $N$ of spherical particles.}
				\label{fig:82}
			\end{subfigure}
			\caption{Computation Times for our Solution Strategies}
		\end{figure}

		\vspace{-3mm}
		
		Equipped with this methodology for picking the FMM parameters $P$ and $D$, we can now compute approximate solutions to the Galerkin discretisation $\eqref{eq:Galerkina}$ for the two cases $\ell_{\max}=5 \text{ and } \ell_{\max}=10$, and an increasing number $N$ of spherical particles. Our goal is to demonstrate numerically that both the GMRES-based and CG-based solution strategies are linear scaling in cost whilst having some reassurance that the FMM error does not dominate the discretisation error for this specific geometric setting.

		All numerical simulations were performed on a 2016 MacBook laptop with a 2.6 GHz Intel Core~i7 processor and 16GB of 2133 MHz LPDDR3 memory. Additionally, we set the linear solver tolerance to $10^{-6}$ and $10^{-9}$ in the case $\ell_{\max}=5$ and $\ell_{\max}=10$ respectively. Our results are displayed in Figures \ref{fig:81} and \ref{fig:82} and indicate excellent agreement with linear scaling behaviour. We see furthermore that the computation times for the CG and GMRES-based approaches are almost identical. Of course, if the geometric setting of the numerical simulations were to be changed, then one linear solver could potentially outperform the other in accordance with the numerical study presented in Section \ref{sec:4.1}.

		\section{Conclusion and Outlook} \label{sec:5}
		This article is the second in a series of two papers in which we present a detailed analysis of a boundary integral equation (BIE) of the second kind that describes the interaction of $N$ dielectric spherical particles undergoing mutual polarisation. The aim of these two articles was to perform a full scalability analysis of the Galerkin method used to solve this BIE and to establish that the method is \emph{linear scaling in accuracy}, i.e., in order to obtain an approximate solution with a fixed relative error, the computational cost of the algorithm scales as~$\mathcal{O}(N)$. In order to show that an $N$-body numerical method is linear scaling in accuracy, it is sufficient to show that it is (a) \emph{$N$-error stable}, i.e., for a fixed number of degrees of freedom per object the {relative} error does not increase with $N$ and (b) linear scaling in cost, i.e., for a fixed number of degrees of freedom per object, only~$\mathcal{O}(N)$ operations are needed to compute an approximate solution. Accordingly, the first article \cite{Hassan} presented the numerical analysis of the Galerkin method and showed that the numerical method is $N$-error stable by deriving $N$-independent convergence rates for the induced surface charge and total electrostatic~energy.
		
		The goal of the current article was to establish that the Galerkin method is also linear scaling in cost. To this end, we presented a convergence analysis of GMRES for the linear system arising from the Galerkin discretisation of the BIE and proved that under mild assumptions, there exists an upper bound-- independent of $N$-- for the number of GMRES iterations required to obtain an approximate solution up to a given relative error tolerance of the Galerkin discretisation. Combined with the use of the FMM to compute matrix-vector products in $\mathcal{O}(N)$, this result establishes that the numerical method is linear scaling in cost. In view of the main result of our earlier work \cite{Hassan}, we can conclude the numerical method is indeed linear scaling in accuracy. 
		
		In addition to our main result, we also demonstrated how to `symmetrise' the underlying linear system and subsequently proposed an equivalent, symmetric linear system that can be solved using the conjugate gradient method (CG) rather than GMRES. A convergence analysis of this alternative approach showed that there exists an upper bound-- also independent of $N$ and qualitatively similar to the GMRES bound-- for the number of CG iterations required to to obtain an approximate solution to the Galerkin discretisation up to a given tolerance.
		Finally, we presented a detailed numerical study with the goal of both supporting our theoretical results and exploring the dependence of the error on various system parameters. 
		
		As regards extensions of this work, we note that the results we have presented thus far involve the induced surface charge as the quantity of interest but many physical applications also require knowledge of the electrostatic forces between the dielectric particles. Similar errors estimates and a linear scaling solution strategy for the electrostatic forces are therefore the subject of a forthcoming article by the authors \cite{Hassan3}. Additionally, the layer potentials and boundary integral operators we have considered thus far are all generated by the Laplace operator. A promising direction of future research could be to explore if similar results on scalability and $N$-independent error estimates hold for boundary integral operators equations arising, for instance, in the study of wave propagation in non-homogenous media or electrostatic interactions between dielectric particles in an ionic solvent.

		\bibliographystyle{plain}
		\bibliography{refs.bib}

\begin{thebibliography}{10}

\bibitem{agullo2014task}
Emmanuel Agullo, B{\'e}renger Bramas, Olivier Coulaud, Eric Darve, Matthias
  Messner, and Toru Takahashi.
\newblock Task-based {FMM} for multicore architectures.
\newblock {\em SIAM Journal on Scientific Computing}, 36(1):C66--C93, 2014.

\bibitem{appel1985efficient}
Andrew Appel.
\newblock An efficient program for many-body simulation.
\newblock {\em SIAM Journal on Scientific and Statistical Computing},
  6(1):85--103, 1985.

\bibitem{barnes1986hierarchical}
Josh Barnes and Piet Hut.
\newblock A hierarchical {O} ({N} log {N}) force-calculation algorithm.
\newblock {\em {N}ature}, 324(6096):446, 1986.

\bibitem{barros2014dielectric}
Kipton Barros and Erik Luijten.
\newblock Dielectric effects in the self-assembly of binary colloidal
  aggregates.
\newblock {\em Physical Review Letters}, 113(1):017801, 2014.

\bibitem{barros2014efficient}
Kipton Barros, Daniel Sinkovits, and Erik Luijten.
\newblock Efficient and accurate simulation of dynamic dielectric objects.
\newblock {\em The Journal of {C}hemical {P}hysics}, 140(6):064903, 2014.

\bibitem{blanchard2015scalfmm}
Pierre Blanchard, B{\'e}renger Bramas, Olivier Coulaud, Eric Darve, Laurent
  Dupuy, Arnaud Etcheverry, and Guillaume Sylvand.
\newblock {S}cal{FMM}: {A} generic parallel fast multipole library.
\newblock In {\em SIAM Conference on Computational Science and Engineering},
  2015.

\bibitem{boateng2013comparison}
Henry Boateng and Robert Krasny.
\newblock Comparison of treecodes for computing electrostatic potentials in
  charged particle systems with disjoint targets and sources.
\newblock {\em Journal of Computational Chemistry}, 34(25):2159--2167, 2013.

\bibitem{Colloid}
Matthias Brunner, Jure Dobnikar, Hans-Hennig von Gr{\"u}nberg, and Clemens
  Bechinger.
\newblock Direct measurement of three-body interactions amongst charged
  colloids.
\newblock {\em Physical {R}eview {L}etters}, 92(7):078301, 2004.

\bibitem{cheng1999fast}
Hongwei Cheng, Leslie Greengard, and Vladimir Rokhlin.
\newblock A fast adaptive multipole algorithm in three dimensions.
\newblock {\em Journal of Computational Physics}, 155(2):468--498, 1999.

\bibitem{Extra2}
Herman Clercx and Georges Bossis.
\newblock Many-body electrostatic interactions in electrorheological fluids.
\newblock {\em Physical Review E}, 48(4):2721, 1993.

\bibitem{dehnen2000very}
Walter Dehnen.
\newblock A very fast and momentum-conserving tree code.
\newblock {\em The Astrophysical Journal Letters}, 536(1):L39, 2000.

\bibitem{dobnikar2002many}
Jure Dobnikar, Y~Chen, Roland Rzehak, and Hans-Hennig von Gr{\"u}nberg.
\newblock Many-body interactions in colloidal suspensions.
\newblock {\em Journal of Physics: Condensed Matter}, 15(1):S263, 2002.

\bibitem{efstathiou1985numerical}
George Efstathiou, Marc Davis, Simon White, and Carlos Frenk.
\newblock Numerical techniques for large cosmological {N}-body simulations.
\newblock {\em The Astrophysical Journal Supplement Series}, 57:241--260, 1985.

\bibitem{eiermann2001geometric}
Michael Eiermann and Oliver Ernst.
\newblock Geometric aspects of the theory of {K}rylov subspace methods.
\newblock {\em Acta Numerica}, 10:251--312, 2001.

\bibitem{eisenstat1983variational}
Stanley Eisenstat, Howard Elman, and Martin Schultz.
\newblock Variational iterative methods for nonsymmetric systems of linear
  equations.
\newblock {\em SIAM Journal on Numerical Analysis}, 20(2):345--357, 1983.

\bibitem{elman1982iterative}
Howard Elman.
\newblock {\em Iterative methods for large, sparse, nonsymmetric systems of
  linear equations}.
\newblock PhD thesis, Yale University New Haven, Connecticut, 1982.

\bibitem{Fischer}
Bernd Fischer.
\newblock {\em Polynomial Based Iteration Methods for Symmetric Linear
  Systems}, volume~68 of {\em Classics in Applied Mathematics}.
\newblock Society for Industrial and Applied Mathematics, Philadelphia,
  Pennsylvania, 2011.

\bibitem{freed2014perturbative}
Karl Freed.
\newblock Perturbative many-body expansion for electrostatic energy and field
  for system of polarizable charged spherical ions in a dielectric medium.
\newblock {\em The Journal of Chemical Physics}, 141(3):034115, 2014.

\bibitem{MR3493124}
Zecheng Gan, Shidong Jiang, Erik Luijten, and Zhenli Xu.
\newblock A hybrid method for systems of closely spaced dielectric spheres and
  ions.
\newblock {\em SIAM Journal on Scientific Computing}, 38(3):B375--B395, 2016.

\bibitem{geng2013treecode}
Weihua Geng and Robert Krasny.
\newblock A treecode-accelerated boundary integral {P}oisson--{B}oltzmann
  solver for electrostatics of solvated biomolecules.
\newblock {\em Journal of Computational Physics}, 247:62--78, 2013.

\bibitem{Greenbaum}
Anne Greenbaum.
\newblock {\em Iterative methods for solving linear systems}, volume~17 of {\em
  Frontiers in Applied Mathematics}.
\newblock Society for Industrial and Applied Mathematics, Philadelphia,
  Pennsylvania, 1997.

\bibitem{greengard2}
Leslie Greengard.
\newblock {\em The rapid evaluation of potential fields in particle systems}.
\newblock ACM Distinguished Dissertations. MIT Press, Cambridge, Massachusetts,
  1988.

\bibitem{greengard1990numerical}
Leslie Greengard.
\newblock The numerical solution of the {N}-body problem.
\newblock {\em Computers in Physics}, 4(2):142--152, 1990.

\bibitem{greengard1}
Leslie Greengard and Vladimir Rokhlin.
\newblock A fast algorithm for particle simulations.
\newblock {\em Journal of {C}omputational {P}hysics}, 73(2):325--348, 1987.

\bibitem{Crystal}
Bartosz Grzybowski, Adam Winkleman, Jason Wiles, Yisroel Brumer, and George
  Whitesides.
\newblock Electrostatic self-assembly of macroscopic crystals using contact
  electrification.
\newblock {\em Nature {M}aterials}, 2(4):241, 2003.

\bibitem{Hassan_Dis}
Muhammad Hassan.
\newblock {\em Mathematical Analysis of Boundary Integral Equations and Domain
  Decomposition Methods with Applications in Polarisable Electrostatics}.
\newblock PhD thesis, RWTH Aachen University, 2020.

\bibitem{Hassan}
Muhammad Hassan and Benjamin Stamm.
\newblock An integral equation formulation of the $ {N} $-body dielectric
  spheres problem. {P}art {I}: {N}umerical {A}nalysis.
\newblock {\em ESAIM: M2AN, Forthcoming article}, 2020.

\bibitem{Hassan3}
Muhammad Hassan and Benjamin Stamm.
\newblock A linear scaling in accuracy numerical method for computing the
  electrostatic forces in the {$N$}-body dielectric spheres problem.
\newblock {\em Accepted for publication in {C}ommunications in {C}omputational
  {P}hysics, arXiv:2002.01579}, 2020.

\bibitem{hockney}
Roger Hockney and James Eastwood.
\newblock {\em Computer simulation using particles}.
\newblock {CRC} press, 1988.

\bibitem{jurrus2018improvements}
Elizabeth Jurrus et~al.
\newblock Improvements to the {APBS} biomolecular solvation software suite.
\newblock {\em {P}rotein {S}cience}, 27(1):112--128, 2018.

\bibitem{knebe2001multi}
Alexander Knebe, Andrew Green, and James Binney.
\newblock Multi-level adaptive particle mesh ({MLAPM}): {A} c code for
  cosmological simulations.
\newblock {\em Monthly Notices of the Royal Astronomical Society},
  325(2):845--864, 2001.

\bibitem{li2009cartesian}
Peijun Li, Hans Johnston, and Robert Krasny.
\newblock A {C}artesian treecode for screened coulomb interactions.
\newblock {\em Journal of Computational Physics}, 228(10):3858--3868, 2009.

\bibitem{Extra1}
Yuncheng Liang, Nidal Hilal, Paul Langston, and Victor Starov.
\newblock Interaction forces between colloidal particles in liquid: {T}heory
  and experiment.
\newblock {\em Advances in Colloid and Interface Science}, 134:151--166, 2007.

\bibitem{TUM}
J{\"o}rg Liesen and Petr Tich{\`y}.
\newblock Convergence analysis of {K}rylov subspace methods.
\newblock {\em GAMM-{M}itteilungen}, 27(2):153--173, 2004.

\bibitem{lindgren2018}
Eric Lindgren, Anthony Stace, Etienne Polack, Yvon Maday, Benjamin Stamm, and
  Elena Besley.
\newblock An integral equation approach to calculate electrostatic interactions
  in many-body dielectric systems.
\newblock {\em Journal of {C}omputational {P}hysics}, 2018.

\bibitem{Titan}
Eric Lindgren, Benjamin Stamm, Ho-Kei Chan, Yvon Maday, Anthony Stace, and
  Elena Besley.
\newblock The effect of like-charge attraction on aerosol growth in the
  atmosphere of {T}itan.
\newblock {\em Icarus}, 291:245--253, 2017.

\bibitem{lindgren2018dynamic}
Eric Lindgren, Benjamin Stamm, Yvon Maday, Elena Besley, and Anthony Stace.
\newblock Dynamic simulations of many-body electrostatic self-assembly.
\newblock {\em Philosophical Transactions of the Royal Society A: Mathematical,
  Physical and Engineering Sciences}, 376(2115):20170143, 2018.

\bibitem{multipole2}
Per Linse.
\newblock Electrostatics in the presence of spherical dielectric
  discontinuities.
\newblock {\em The Journal of {C}hemical {P}hysics}, 128(21):214505, 2008.

\bibitem{lotan2006}
Itay Lotan and Teresa Head-Gordon.
\newblock An analytical electrostatic model for salt screened interactions
  between multiple proteins.
\newblock {\em Journal of Chemical Theory and Computation}, 2(3):541--555,
  2006.

\bibitem{Self}
Logan McCarty, Adam Winkleman, and George Whitesides.
\newblock Electrostatic self-assembly of polystyrene microspheres by using
  chemically directed contact electrification.
\newblock {\em {A}ngewandte {C}hemie {I}nternational {E}dition},
  46(1-2):206--209, 2007.

\bibitem{merrill2009many}
Jason Merrill, Sunil Sainis, and Eric Dufresne.
\newblock Many-body electrostatic forces between colloidal particles at
  vanishing ionic strength.
\newblock {\em Physical Review Letters}, 103(13):138301, 2009.

\bibitem{image1}
Ren{\'e} Messina.
\newblock Image charges in spherical geometry: {A}pplication to colloidal
  systems.
\newblock {\em The {J}ournal of {C}hemical {P}hysics}, 117(24):11062--11074,
  2002.

\bibitem{messner2012optimized}
Matthias Messner, B{\'e}renger Bramas, Olivier Coulaud, and Eric Darve.
\newblock Optimized {M}2{L} kernels for the {C}hebyshev interpolation based
  fast multipole method.
\newblock {\em arXiv preprint arXiv:1210.7292}, 2012.

\bibitem{nachtigal1992fast}
No{\"e}l Nachtigal, Satish Reddy, and Lloyd Trefethen.
\newblock How fast are nonsymmetric matrix iterations?
\newblock {\em SIAM Journal on {M}atrix {A}nalysis and {A}pplications},
  13(3):778--795, 1992.

\bibitem{polymer}
Herbert Pohl.
\newblock Giant polarization in high polymers.
\newblock {\em Journal of Electronic Materials}, 15(4):201--203, 1986.

\bibitem{qin2016image}
Jian Qin, Juan de~Pablo, and Karl Freed.
\newblock Image method for induced surface charge from many-body system of
  dielectric spheres.
\newblock {\em The Journal of Chemical Physics}, 145(12):124903, 2016.

\bibitem{image3}
Jian Qin, Jiyuan Li, Victor Lee, Heinrich Jaeger, Juan de~Pablo, and Karl
  Freed.
\newblock A theory of interactions between polarizable dielectric spheres.
\newblock {\em Journal of {C}olloid and {I}nterface {S}cience}, 469:237--241,
  2016.

\bibitem{saad1981krylov}
Yousef Saad.
\newblock Krylov subspace methods for solving large unsymmetric linear systems.
\newblock {\em Mathematics of Computation}, 37(155):105--126, 1981.

\bibitem{Saad}
Yousef Saad.
\newblock {\em Iterative methods for sparse linear systems}.
\newblock Society for Industrial and Applied Mathematics, Philadelphia,
  Pennsylvania, second edition, 2003.

\bibitem{saad1986}
Yousef Saad and Martin Schultz.
\newblock {GMRES}: A generalized minimal residual algorithm for solving
  nonsymmetric linear systems.
\newblock {\em SIAM Journal on Scientific and Statistical Computing},
  7(3):856--869, 1986.

\bibitem{Schwab}
Stefan Sauter and Christoph Schwab.
\newblock {\em Boundary element methods}, volume~39 of {\em Springer Series in
  Computational Mathematics}.
\newblock Springer-Verlag, Berlin, 2011.

\bibitem{Lattice}
Elena Shevchenko, Dmitri Talapin, Nicholas Kotov, Stephen O'Brien, and
  Christopher Murray.
\newblock Structural diversity in binary nanoparticle superlattices.
\newblock {\em Nature}, 439(7072):55, 2006.

\bibitem{image2}
Zhenli Xu.
\newblock Electrostatic interaction in the presence of dielectric interfaces
  and polarization-induced like-charge attraction.
\newblock {\em Physical {R}eview {E}}, 87(1):013307, 2013.

\bibitem{yap2013calculating}
Eng-Hui Yap and Teresa Head-Gordon.
\newblock Calculating the bimolecular rate of protein--protein association with
  interacting crowders.
\newblock {\em Journal of {C}hemical {T}heory and {C}omputation},
  9(5):2481--2489, 2013.

\end{thebibliography}

	\end{document}